\documentclass{aims}
\usepackage{bbm} 
\usepackage{mathrsfs} 
\usepackage{color}  
\usepackage{amsmath}
  \usepackage{paralist}
  \usepackage{graphics} 
  \usepackage{epsfig} 
\usepackage{graphicx}  \usepackage{epstopdf}
 \usepackage[colorlinks=true]{hyperref}
\hypersetup{urlcolor=blue, citecolor=red}

\usepackage{enumitem} 

\newtheorem{theorem}{Theorem}[section]
\newtheorem{corollary}{Corollary}

\newtheorem{lemma}[theorem]{Lemma}
\newtheorem{proposition}{Proposition}

\theoremstyle{definition}
\newtheorem{definition}[theorem]{Definition}

\newtheorem*{BA}{Basic Assumption}  
\newtheorem{remark}{Remark}

\newtheorem*{exa*}{Example} 
\newtheorem{theoremm}{Theorem}  
\newtheorem{corollaryy}{Corollary}[theoremm]
\newtheorem*{pf}{Proof}
\newtheorem*{potA}{Proof of Theorem \ref{thm:localvsfinite}}
\newtheorem*{potB}{Proof of Theorem \ref{thm:SMB}}
\newtheorem*{potC}{Proof of Theorem \ref{thm:vp}}

\title[Unstable Entropy and Unstable Pressure for RPHDS] 
      {Unstable Entropy and Unstable Pressure for Random Partially Hyperbolic Dynamical Systems}

\author[Xinsheng Wang Weisheng Wu and Yujun Zhu]{}


 \email{xswang@hebtu.edu.cn}
 \email{wuweisheng@cau.edu.cn}
 \email{yjzhu@xmu.edu.cn}

\thanks{X. Wang and Y. Zhu are supported by  NSFC (No: 11771118, 11801336), W. Wu is supported by  NSFC (No: 11701559). The first author is also supported by the Innovation Fund Designated for Graduate Students of Hebei Province (No: CXZZBS2018101) and China Scholarship Council (CSC)}

\thanks{$^*$ The corresponding author.}

\begin{document}

\maketitle

\centerline{\scshape Xinsheng Wang$^{1,*}$ Weisheng Wu$^2$ and Yujun Zhu$^{3}$}

\medskip

 {\footnotesize
 \centerline{1. School of Mathematical Sciences}
   \centerline{ Hebei Normal University, Shijiazhuang, 050024, P.R. China}
   \centerline{2. Department of Applied Mathematics, College of Science}
   \centerline{ China Agricultural University, Beijing, 100083, P.R. China}
   \centerline{3. School of Mathematical Sciences}
   \centerline{ Xiamen University, Xiamen, 361005, P.R. China}
} 

\bigskip

\begin{abstract}
Let $\mathcal{F}$ be a $C^2$ {random partially hyperbolic dynamical system}. For the unstable foliation, the corresponding unstable metric entropy, unstable topological entropy and unstable pressure via the dynamics of $\mathcal{F}$ on the unstable foliation are introduced and investigated. A version of Shannon-McMillan-Breiman Theorem for unstable metric entropy is given, and a variational principle for unstable pressure (and hence for unstable entropy) is obtained. Moreover, as an application of the variational principle, equilibrium states for the unstable pressure including Gibbs $u$-states are investigated.
\end{abstract}

\renewcommand{\sectionmark}[1]{}
\section{Introduction}

In differentiable dynamical systems and smooth ergodic theory, entropies (including metric entropy and topological entropy), pressures and Lyapunov exponents are the main ingredients for both of deterministic and random cases, which describe the complexity of the orbit structure of the system from different points of view.

In the seminal papers \cite{LedrappierYoung1985} and \cite{LedrappierYoung1985a}, the so-called Pesin's entropy formula and dimension formula which relate metric entropy and positive Lyapunov exponents are given for any $C^2$ diffeomorphism $f$ on a closed manifold $M$ with respect to SRB measures and general invariant measures respectively. These formulas tell us that positive exponents have contribution to the metric entropy $h_\mu(f)$, where $\mu$ is an $f$-invariant measure. In another word, $h_\mu(f)$ is determined by the dynamics of $f$ on the unstable foliations since it can be given by $H_\mu(\xi|f\xi)$, where $\xi$ is an increasing partition subordinate to unstable manifolds. An interesting question is: {\em can we introduce an appropriate definition of topological entropy $h^u_{\text{top}}(f)$ via the dynamics of $f$ on the unstable foliations and obtain a version of variational principle relating $h^u_{\text{top}}(f)$ and $h_\mu(f)$?}

Recently, a partial answer to the above question is obtained, and the theory of entropy and pressure along unstable foliations for $C^1$ partially hyperbolic diffeomorphisms is investigated (cf.\cite{HuHuaWu2017}, \cite{Yang2016} and \cite{HuWuZhu2017}). In \cite{HuHuaWu2017},  Hu, Hua and Wu introduce the definitions of unstable metric entropy $h^u_{\mu}(f)$ for  any invariant measure $\mu$ and unstable topological entropy $h_{\text{top}}^u(f)$,  give a version of Shannon-McMillan-Breiman Theorem for $h^u_{\mu}(f)$, and obtain a variational principle relating $h^u_{\mu}(f)$ and $h_{\text{top}}^u(f)$. We point out that in \cite{HuHuaWu2017} $h^u_{\mu}(f)$ is defined via
$H_{\mu}(\bigvee_{i=0}^{n-1}f^{-i}\alpha|\eta)$ (where $\alpha$ is a finite measurable partition and $\eta$ is a measurable
partition which is subordinate to unstable manifolds) instead of  the classical form $H_{\mu}(\xi|f\xi)$ (where $\xi$ is an increasing partition subordinate to the unstable manifolds) in \cite{LedrappierYoung1985a}. The advantage of this type of definition of  $h^u_{\mu}(f)$ is that a variational principle then is obtained and hence the theory of pressure and  related topics in mathematical statistical mechanics can be considered. Actually, lately in \cite{HuWuZhu2017}, Hu, Wu and Zhu generalize the concept of unstable topological entropy to unstable pressure $P^u(f, \varphi)$ for any continuous function $\varphi$ on $M$, obtain a variational principle for unstable pressure, and investigate the properties of the so-called $u$-equilibriums. In fact, we observe that for any $C^1$ diffeomorphism $f$ with uniformly expanding foliations (see the precise definition in \cite{Yang2016}), unstable entropies $h^u_{\mu}(f)$, $h^u_{\text{top}}(f)$ and unstable pressure $P^u(f, \varphi)$ can be defined via the dynamics of $f$ on the unstable foliations, and moreover, a sequence of similar results as in \cite{HuHuaWu2017} and \cite{HuWuZhu2017} can be obtained with an effort.

The main purpose of this paper is to consider this topic for random dynamical systems (RDSs). In \cite{Kifer1986} and \cite{LiuQian1995},  Kifer and Liu  are mainly concerned about  i.i.d. (i.e., independent and identically distributed) RDSs generated by applying at each step a transformation chosen randomly from a given family according to some probability distribution. Specifically, for a $C^2$ i.i.d. RDS $\mathcal{F}$, the ergodic theory, which mainly consists of Pesin theory, of $\mathcal{F}$ has been systematically investigated in \cite{Kifer1986} and \cite{LiuQian1995}. In \cite{Arnold1998}, Arnold considers a more general version of RDSs, where the random choice of transformations is assumed to be only stationary. We are mainly concerned about the general case in this paper. For RDSs, various random versions of Pesin's entropy formula and dimension formula were thoroughly investigated in \cite{LedrappierYoung1988}, \cite{LedrappierYoung1988a}, \cite{LiuQian1995}, \cite{BahnmullerLiu1998}, \cite{QianQianXie2003}, etc., in different settings. In these papers, the metric entropy $h_{\mu}(\mathcal{F})$ is defined by $H_{{\mu}}(\xi|\Theta\xi)$, where $\Theta$ is the induced skew product transformation on $\Omega\times M$ ($\Omega$ is a probability space with probability $\mathbf{P}$), ${\mu}$ is a $\Theta$-invariant measure with marginal measure $\mathbf{P}$ on $\Omega$ and $\xi$ is an increasing partition of $\Omega\times M$ which is subordinate to the random unstable foliations (see Section \ref{sec:Pre} for notations and definitions). In this paper we will adapt the techniques in \cite{HuHuaWu2017} and \cite{HuWuZhu2017} to the random setting and obtain the corresponding results as in \cite{HuHuaWu2017} and \cite{HuWuZhu2017}.

We will firstly introduce the so-called partial hyperbolicity for $\mathcal{F}$ (see Definition \ref{domination}). Then for the unstable foliation, we define two types of unstable metric entropies $\tilde{h}^u_{\mu}(\mathcal{F})$ (via Bowen balls) and ${h}^u_{\mu}(\mathcal{F})$ (via conditional entropy) with respect to any $\mathcal{F}$-invariant measure $\mu$ following the methods in \cite{LedrappierYoung1985a} and \cite{HuHuaWu2017} (see Definition \ref{def:metricentropy1} and Definition \ref{def:metricentropy2}), and show that these two unstable metric entropies coincide with each other when  $\mu$ is ergodic (Theorem \ref{thm:localvsfinite}). Some properties of  unstable metric entropies are given and a version of Shannon-McMillan-Breiman Theorem, in which the unstable metric entropy is expressed as the limit of certain conditional information functions, is obtained (Theorem \ref{thm:SMB}). The next work is, for the unstable foliation, to define unstable pressure $P^u(\mathcal{F}, \phi)$  for each function $\phi$ on $\Omega\times M$ which is continuous in $x\in M$ and measurable in $\omega\in\Omega$ via the dynamics of $\mathcal{F}$ on local unstable manifolds (see Definition \ref{def:upressure}). Then a version of variational principle for $P^u(\mathcal{F}, \phi)$, which states that $P^u(\mathcal{F}, \phi)$ is the supremum of the sum of the unstable metric entropy and the integral of $\phi$ taken over all invariant measures of $\mathcal{F}$, is obtained (Theorem \ref{thm:vp}). Since $P^u(\mathcal{F}, 0)=h^u_{\text{top}}(\mathcal{F})$ (the unstable topological entropy of $\mathcal{F}$), we get a variational principle for  $h^u_{\text{top}}(\mathcal{F})$ directly (Corollary \ref{coro:vp1}).

This paper is organized as follows. In Section \ref{sec:Pre}, we give some preliminaries and state our main results. In Section \ref{sec:Meaentropy}, we give the definition of unstable metric entropy via two methods and obtain the equivalence of them under an ergodicity condition. In Section \ref{sec:SMBThm}, a version of Shannon-McMillan-Breiman Theorem for unstable metric entropy is given. In Section \ref{sec:Pressure}, we give the definition of unstable pressure. In Section \ref{sec:VP}, a variational principle for unstable pressure is obtained. In the last section, i.e. Section \ref{sec:equilibrium}, as an application of the variational principle, we discuss the so-called $u$-equilibrium states. Particularly, the so-called Gibbs $u$-state are introduced.

\section{Preliminaries and Statement of Main Results}\label{sec:Pre}

Throughout this paper, we let $M$ be a $C^\infty$ compact Riemannian manifold without boundary. Denote by $\mathscr{B}(M)$ the Borel $\sigma$-algebra of $M$. Let $(\Omega,\mathscr{F},\mathbf{P})$ be a Polish probability space and $\theta$ be an invertible and ergodic measure-preserving transformation on $\Omega$.

\begin{definition}
A {\em $C^2$ random dynamical system} $\mathcal{F}$ on $M$ over $(\Omega,\mathscr{F},\mathbf{P},\theta)$ is defined as a map
\begin{align*}
  \mathcal{F}\colon\mathbb{Z}\times\Omega\times M &\to M  \\
  (n,\omega,x) &\mapsto \mathcal{F}(n,\omega)x,
\end{align*}
which has the following properties:
\begin{enumerate}[label=(\roman*)]
  \item $\mathcal{F}$ is measurable;
  \item the maps $\mathcal{F}(n,\omega)\colon M\to M$ form a cocycle over $\theta$, i.e. they satisfy
  \begin{align*}
    \mathcal{F}(0,\omega) &=\mathrm{id}, \\
    \mathcal{F}(n+m,\omega) &=\mathcal{F}(m,\theta^n\omega)\circ\mathcal{F}(n,\omega),
  \end{align*}
  for all $n,m\in\mathbb{Z}$ and $\omega\in\Omega$;
  \item the maps $\mathcal{F}(n,\omega)\colon M\to M$ are $C^2$ for all $n\in\mathbb{Z}$ and $\omega\in\Omega$.
\end{enumerate}
\end{definition}

For each $\omega\in\Omega$, we define
\[
  f^n_\omega:=
  \begin{cases}
    \mathcal{F}(1,\theta^{n-1}\omega)\circ\cdots\circ \mathcal{F}(1,\omega) &\text{if }n>0,\\
    \mathrm{id} &\text{if }n=0, \\
    \mathcal{F}(1,\theta^n\omega)^{-1}\circ\cdots\circ \mathcal{F}(1,\theta^{-1}\omega)^{-1} &\text{if }n<0.
  \end{cases}
\]

  Associated with $\Omega\times M$, there is a skew product $\Theta$ induced by $\mathcal{F}$, i.e.,
  \begin{align*}
    \Theta\colon\Omega\times M &\to\Omega\times M \\
    (\omega,x) &\mapsto(\theta\omega,\mathcal{F}(1,\omega)x).
  \end{align*}

\begin{definition}[Invariant measure]
  A measure $\mu$ on $\Omega\times M$ is said to be an {\em$\mathcal{F}$-invariant measure}, if it is $\Theta$-invariant and has marginal measure $\mathbf{P}$ on $\Omega$. In particular, an $\mathcal{F}$-invariant measure $\mu$ is said to be {\em ergodic}, if it is ergodic with respect to $\Theta$.
\end{definition}
 We denote by $\mathcal{M}_{\mathbf{P}}(\mathcal{F})$ the set of all $\mathcal{F}$-invariant measures.

In the following part of this paper, for $\mu\in\mathcal{M}_{\mathbf{P}}(\mathcal{F})$ we always consider the ${\mu}$-completion of $\mathscr{F}\times\mathscr{B}(M)$, which is still denoted by $\mathscr{F}\times\mathscr{B}(M)$ for simplicity.

According to Rokhlin's paper \cite{Rokhlin1962} and Liu and Qian's monograph\cite{LiuQian1995}, for each $\mu\in\mathcal{M}_{\mathbf{P}}(\mathcal{F})$, there exists a family of sample measures ${\mu}_\cdot(\cdot)\colon\Omega\times\mathscr{B}(M)\to [0,1]$ of ${\mu}$ satisfying the following properties:
\begin{enumerate}[label=(\roman*)]
  \item for all $B\in\mathscr{B}(M)$, $\omega\mapsto{\mu}_\omega(B)$ is $\mathscr{F}$-measurable;
  \item for $\mathbf{P}$-a.e. $\omega\in\Omega$, ${\mu}_\omega\colon\mathscr{B}(M)\to[0,1]$ is a probability measure on $M$;
  \item for $A\in\mathscr{F}\times\mathscr{B}(M)$,
  \[
    {\mu}(A)=\int_\Omega\int_{M}\mathbf{1}_A(\omega,x)\mathrm{d}{\mu}_\omega(x)\mathrm{d}{\mathbf{P}}(\omega),
  \]
  where $\mathbf{1}_A$ is the characteristic function of $A\subset\Omega\times M$.
\end{enumerate}
\begin{remark}
  For $\mu\in\mathcal{M}_{\mathbf{P}}(\mathcal{F})$, it is clear that $f_\omega^n(\omega){\mu}_\omega={\mu}_{\theta^n\omega}$ for all $n$ $\in\mathbb{Z}$ and $\mathbf{P}$-a.e. $\omega\in\Omega$.
\end{remark}

Throughout this paper, we always assume that the Probability $\mathbf{P}$ on $\Omega$ satisfies
\begin{equation}\label{eq:basic}
  \int_\Omega(\log^+|\mathcal{F}(1,\omega)|_{C^2}+\log^+|\mathcal{F}(-1,\omega)|_{C^2})\mathrm{d}\mathbf{P}(\omega)<\infty,
\end{equation}
where $|f|_{C^2}$ denotes the usual $C^2$ norm of $f\in\mathrm{Diff}^2(M)$, and $\log^+a=\max\{\log a, 0\}$.

Similar to the deterministic case, we can define the Lyapunov exponents for $\mathcal{F}$. Let $\Lambda$ be the set of all regular points $(\omega,x)\in\Omega\times M$ in the sense of Oseledec. For $(\omega,x)\in\Lambda$, Let $\lambda_1(\omega,x)>\cdots>\lambda_{r(\omega,x)}(\omega,x)$ be its distinct Lyapunov exponents of $\mathcal{F}$ with multiplicities $m_j(\omega,x)$ ($1\leq j\leq r(\omega,x)$).

Let $\mu\in\mathcal{M}_{\mathbf{P}}(\mathcal{F})$, and denote by $\Vert\cdot\Vert$ the norm of vectors in the tangent space {of $M$}. By the Oseledec Multiplicative Ergodic Theorem {(refer to Chapter 4 in \cite{Arnold1998} for more details)}, we know that $\Lambda$ is $\Theta$-invariant and ${\mu}(\Lambda)=1$. For each $(\omega,x)\in\Lambda$, there is a splitting of $T_xM$ as follows
\[
  T_xM=E_1(\omega,x)\oplus\cdots\oplus E_{r({\omega,}x)}(\omega,x)
\]
such that for $i=1,\dots,r({\omega,}x)$, $\mathrm{dim}E_i(\omega,x)=m_i({\omega,}x)$ and
\[
  \lim_{n\to\pm\infty}\frac{1}{n}\log\Vert \mathrm{D}_xf^n_\omega v\Vert=\lambda_i({\omega,}x),\quad\text{for all } v\in E_i(\omega,x)\setminus\{0\}.
\]
{(By \eqref{eq:basic} we can choose $\Lambda$ such that all the above Lyapunov exponents are finite.)}
Let
\[
  u({\omega,}x)=\max\{j\colon\lambda_j({\omega,}x)>0\}.
\]
For $(\omega,x)\in\Lambda$, we define the set
\[
  W^u(\omega,x)=\{y\in M\colon\limsup_{n\to+\infty}\frac{1}{n}\log d(f^{-n}_\omega y,f^{-n}_\omega x)\leq-\lambda_{u({\omega,}x)}({\omega,}x)\},
\]
where $d(\cdot,\cdot)$ is the metric on $M$ induced by its Riemannian structure.
Let
\[
  E^u(\omega,x)=\oplus_{j=1}^{u({\omega,}x)}E_j(\omega,x)\text{, }\quad F^u(\omega,x)=\oplus_{j=u({\omega,}x)+1}^{r({\omega,}x)}E_j(\omega,x).
\]
It is clear that both $E^u(\omega,x)$ and $F^u(\omega,x)$ are invariant under the tangent map, i.e. for $n\in\mathbb{Z}$,
\[
  \mathrm{D}_xf_\omega^nE^u(\omega,x)=E^u(\theta^n\omega,f_\omega^nx)\quad\text{and}\quad\mathrm{D}_xf_\omega^nF^u(\omega,x)=F^u(\theta^n\omega,f_\omega^nx).
\]

The following proposition from \cite{BahnmullerLiu1998} ensures that $W^u(\omega,x)$ is an immersed submanifold of $M$.
\begin{proposition}\label{prop:unstablemanifold}
  For $(\omega,x)\in\Lambda$, the set $W^u(\omega,x)$ is a $C^{1,1}$ immersed submanifold of $M$ tangent at $x$ to $E^u(\omega,x)$.
\end{proposition}
{When $(\omega,x)\in\Omega\times M\setminus\Lambda$, we let $W^u(\omega,x)=\{x\}$, making the definitions of unstable topological entropy in subsequent sections more transparent.}We call the collection $\{W^u(\omega,x):(\omega,x)\in\Omega\times M\}$ \emph{$W^u$-foliation}.

For a measurable partition $\alpha$ of $\Omega\times M$, and $\omega\in\Omega$, define
\[
  \alpha_\omega(x):=\{ y:(\omega,y)\in\alpha(\omega,x) \},
\]
where $\alpha(\omega,x)$ is the element of $\alpha$ containing $(\omega,x)$. It is clear that for $x$, $x'\in M$, either $\alpha_\omega(x)=\alpha_\omega(x')$ or $\alpha_\omega(x)\cap\alpha_\omega(x')=\emptyset$, so $\alpha_\omega\colon=\{\alpha_\omega(x)\}$ is a partition of $M$.

A measurable partition $\xi$ of $\Omega\times M$ with $\xi\geq\sigma_0$, where $\sigma_0$ is the trivial partition $\{ \{ \omega \}\times M: \omega\in\Omega \}$, is said to \emph{be subordinate to the $W^u$-foliation}, if for ${\mu}$-a.e. $(\omega,x)\in\Omega\times M$, $\xi_\omega(x)\subset W^u(\omega,x)$ and it contains an open neighborhood of $x$ in $W^u(\omega,x)$.

For each measurable partition $\eta$ subordinate to $W^u$-foliation, there is a canonical system of conditional measures $\{{\mu}^\eta_{(\omega,x)}\}_{(\omega,x)\in\Omega\times M}$ of  ${\mu}$ by a classical result of Liu and Qian \cite{LiuQian1995}. And ${\mu}^\eta_{(\omega,x)}$ can be regarded as a measure on $\eta_\omega(x)$, if we identify $\{\omega\}\times\eta_\omega(x)$ with $\eta_\omega(x)$.

Now we give the definition of {random partially hyperbolic dynamical systems}. A random variable $t:\Omega\longrightarrow \mathbb{R}$ is called \emph{$\theta$-invariant} if $t(\theta\omega)=t(\omega)$ for $\mathbf{P}$-a.e. $\omega\in \Omega$.

\begin{definition}\label{domination}
 A $C^2$ RDS $\mathcal{F}$ is said to be {\em partially hyperbolic}, if
 \begin{enumerate}[fullwidth,label=(\roman*)]
 \item there exist {$\theta$-invariant variables $C$ and $\lambda$ on $\Omega$} satisfying for $\mathbf{P}$-a.e. $\omega\in\Omega$
  \[
    C(\omega)>1\quad\text{and}\quad0<\lambda(\omega)<1,
  \]
 {and such that} for all $(\omega,x)\in \Lambda$, $n\in\mathbb{N}$, $u\in F^u(\omega,x)\setminus\{0\}$ and $v\in E^u(\omega,x)\setminus\{0\}$,
  \[
    \frac{\Vert\mathrm{D}f^n_\omega u\Vert}{\Vert u\Vert}\leq C(\omega)(\lambda(\omega))^n\frac{\Vert\mathrm{D}f^n_\omega v\Vert}{\Vert v\Vert};
  \]
 \item\label{uniformexpansion} $\mathcal{F}$ is uniformly expanding in $x$ along $W^u$-foliation, {i.e., there exists a $\theta$-invariant random variable $\tilde{\lambda}(\omega)>1$ such that $\Vert Df^1_\omega|E^u(\omega,x)\Vert>\tilde{\lambda}(\omega)$.}
 \end{enumerate}
\end{definition}

\begin{remark}\label{rmk:domination}
  \begin{enumerate}[fullwidth,label=(\roman*)]
    \item In fact, $\tilde\lambda(\omega)$ is a.e. constant, since $\theta$ is $\mathbf{P}$-ergodic.
    \item If \ref{uniformexpansion} in Definition \ref{domination} does not hold, $\mathcal{F}$ is called a $C^2$ RDS with $u$-domination {(cf. \cite{WangWangZhu2017} for deterministic case)}. Notice that the uniform expansion is a crucial property for our proofs of the main results. {We think} similar results {should} hold for $C^2$ RDSs with $u$-domination, {but more complicated techniques involving Pesin theory must be applied.}
  \end{enumerate}
\end{remark}

\begin{BA}
  In the remaining of this paper, we always assume that $\mathcal{F}$ is a $C^2$ {random partially hyperbolic dynamical system}.
\end{BA}

\begin{exa*}
Let $f$ be a $C^2$ partially hyperbolic diffeomorphism. Combining the techniques in \cite{Liu1998} and \cite{HuZhu2014}, we can obtain a random partially hyperbolic dynamical system satisfying the above assumption via small $C^2$ random perturbations.
\end{exa*}

We call a {measurable partition $\alpha$ of $\Omega\times M$ is \emph{fiberwise finite} if for $\mathbf{P}$-a.e. $\omega\in\Omega$, $\alpha_\omega$ is finite and
\[
  \int_{\Omega}K(\omega)\mathrm{d}\mathbf{P}<\infty,
\]
where $K(\omega)$ is the cardinality of $\alpha_\omega$.}

For a fiberwise finite partition $\alpha$ of $\Omega\times M$ and $\omega\in\Omega$, define $\mathrm{diam}(\alpha_\omega)$ as follows
\[
  \mathrm{diam}(\alpha_\omega):=\max_{A\in\alpha_\omega}\mathrm{diam}(A),
\]
where $\mathrm{diam}(A)=\sup_{x,y\in A}d(x,y)$. The variable $\mathrm{diam}(\alpha_\cdot)\colon\Omega\to\mathbb{R}^+$ is defined as the {\em diameter} of $\alpha$.

Let $\alpha$ be a fiberwise finite partition of $\Omega\times M$ with diameter small enough. By such $\alpha$ we can construct a measurable partition as follows. Define
{\[
  \eta=\{\alpha(\omega,x)\cap(\{\omega\}\times W^u_{\text{loc}}(\omega,x)):(\omega,x)\in\Omega\times M\},
\]
where $W^u_{\text{loc}}(\omega,x)$ is a local unstable manifold at $(\omega,x)$ whose size is greater than the diameter of $\alpha_\omega$.} Denote by $\mathcal{P}(\Omega\times M)$ and $\mathcal{P}^u(\Omega\times M)$ the set of all fiberwise finite partitions and the set of all partitions constructed as above respectively.
\begin{remark}
  It is easy to check that if for $\mathbf{P}$-a.e. $\omega$, ${\mu}_\omega(\partial \alpha_\omega)=0$, where $\partial \alpha_\omega=\bigcup_{A\in\alpha_\omega}\partial A$, {then a} measurable partition described as above is a partition subordinate to the $W^u$-foliation.
\end{remark}

The following proposition ensures the existence of another class of useful partitions.
\begin{proposition}[Cf. Proposition 3.7, \cite{BahnmullerLiu1998}]\label{prop:specialpartition}
  Let $\mathcal{F}$ be a $C^2$ RDS. Then there exists a measurable partition $\xi_u$ of $\Omega\times M$ satisfying the following properties:
  \begin{enumerate}[label=(\roman*)]
  	\item $\xi_u$ is increasing, i.e., $\Theta^{-1}\xi_u\geq\xi_u$;
  	\item $\xi_u$ is subordinate to the $W^u$-foliation;
  	\item $\bigvee_{n=1}^{+\infty}\Theta^{-n}\xi_u=\varepsilon$, where $\varepsilon$ is the partition of $\Omega\times M$ into points, i.e. $\xi_u$ is a generator;
  	\item $\mathscr{B}(\bigwedge_{n=1}^{+\infty}\Theta^n\xi_u)=\mathscr{B}^u(\mathcal{F})$, where for a measurable partition  $\xi$ of $\Omega\times M$, $\mathscr{B}(\xi)$ is the $\sigma$-algebra generated by $\xi$ and $\mathscr{B}^u(\mathcal{F})= \{B\in\mathscr{F}\times\mathscr{B}(M)\colon(\omega,x)\in B\text{ implies }\{\omega\}\times W^u(\omega,x)\subset B\}$.
  \end{enumerate}
\end{proposition}

A partition is called an \emph{increasing partition subordinate to {$W^u$-foliation}} if it satisfies Proposition \ref{prop:specialpartition}, and denote by $\mathcal{Q}^u(\Omega\times M)$ the set of all such partitions.

As the deterministic case, we can define two types of the unstable metric entropies along $W^u$-foliation for $\mathcal{F}$. The first one is defined via the average decreasing rate of the measure of the Bowen balls (as in \cite{QianQianXie2003}), and we denote it by $\tilde{h}^u_\mu(\mathcal{F})$.  The second one is defined via the conditional entropy of $\mathcal{F}$ along $W^u$-foliation (adapted from  \cite{HuHuaWu2017} and \cite{WangWangZhu2017}), which we denote by ${h}^u_\mu(\mathcal{F})$. Their precise definitions are given {in Section \ref{sec:Meaentropy}}.

In fact, when $\mu$ is ergodic, the entropies described above are equivalent, i.e., we have the following theorem.
\begin{theoremm}\label{thm:localvsfinite}
Let $\mu$ be an ergodic measure of $\mathcal{F}$. Then
  \[
    \tilde{h}^u_\mu(\mathcal{F})={h}^u_\mu(\mathcal{F}).
  \]
\end{theoremm}

For measurable partitions $\beta$ and $\gamma$ of $\Omega\times M$, we can define the conditional information function and conditional entropy of $\beta$ {with respect to $\gamma$} for an invariant measure ${\mu}$ of $\Theta$, which are denoted by $I_{{\mu}}(\beta|\gamma)$ and $H_{{\mu}}(\beta|\gamma)$ respectively. These notations are standard, and more details are described at the end of this section. The conditional entropy of $\mathcal{F}$ for measurable partition $\beta$ with respect to $\gamma$ can also be given, which is denoted by $h_\mu(\mathcal{F},\beta|\gamma)$ and described in Section \ref{sec:Meaentropy}. For integers $0\leq k<j$, denote $\beta_k^j=\Theta^{-j}\beta\vee\Theta^{-(j-1)}\beta\vee\cdots\vee\Theta^{-k}\beta$. A version of Shannon-McMillan-Breiman Theorem in our case is also obtained as follows.
\begin{theoremm}\label{thm:SMB}
 Suppose $\mu$ is an ergodic measure of $\mathcal{F}$. {For any $\alpha\in \mathcal{P}(\Omega\times M), \eta\in\mathcal{P}^u(\Omega\times M)$}, we have
  \[
    \lim_{n\to\infty}\frac{1}{n}I_{{\mu}}(\alpha^{n-1}_0|\eta)(\omega,x)=h_\mu(\mathcal{F},\alpha|\eta).
  \]
\end{theoremm}

We can also define unstable topological entropy and unstable pressure for a potential function $\phi$, whose precise definitions are described in Section \ref{sec:Pressure}, and we denote them by $h_{\text{top}}^u(\mathcal{F})$ and $P^u(\mathcal{F},\phi)$ respectively. Naturally, a variational principle is formulated in the following. Denote the set
\[
  \{\phi\in L^1(\Omega\times M):\phi\text{ is measurable in }\omega\text{, continuous in }x\}
\]
by $L^1(\Omega,C(M))$.
\begin{theoremm}\label{thm:vp}
  Let $\phi\in L^1(\Omega,C(M))$. Then we have
  \[
    \sup_{\mu\in\mathcal{M}_{\mathbf{P}}(\mathcal{F})}\left\{h^u_\mu(\mathcal{F})+\int_{\Omega\times M}\phi\mathrm{d}{\mu}\right\}=P^u(\mathcal{F},\phi).
  \]
\end{theoremm}

A direct corollary of Theorem \ref{thm:vp} is the following variational principle for unstable entropy.
\begin{corollaryy}\label{coro:vp1}
 \[
   \sup_{\mu\in\mathcal{M}_{\mathbf{P}}(\mathcal{F})}\left\{h^u_\mu(\mathcal{F})\right\}=h_{\text{top}}^u(\mathcal{F}).
 \]
\end{corollaryy}

In the remaining of this section, we give some knowledge on information function, which is slightly modified in our context.
{Recall that for a measurable partition $\eta$ of a measure space $X$
and a probability measure $\nu$ on $X$, the \emph{canonical system
of conditional measures for $\nu$ and $\eta$} is a family of probability
measures $\{\nu_x^\eta: x\in X\}$ with $\nu_x^\eta\bigl(\eta(x)\bigr)=1$,
such that for every measurable set $B\subset X$, $x\mapsto \nu_x^{\eta}(B)$
is measurable and
\[
\nu (B)=\int_X\nu_x^{\eta}(B)d\nu(x).
\]
See e.g. \cite{Rokhlin1962} for more details.

\begin{definition}
  Let ${\mu}$ be an invariant measure of $(\Omega\times M,\Theta)$, $\alpha$ and $\eta$ be two measurable partitions of $\Omega\times M$. \emph{The information function} of $\alpha$ with respect to ${\mu}$ is defined as
  \[
    I_{{\mu}}(\alpha)(\omega,x):=-\log{\mu}(\alpha(\omega,x)),
  \]
  and the \emph{entropy} of $\alpha$ with respect to ${\mu}$ is defined as
  \[
    H_{{\mu}}(\alpha):=\int_{\Omega\times M}I_{{\mu}}(\alpha)(\omega,x)\mathrm{d}{\mu}(\omega,x)=-\int_{\Omega\times M}\log{\mu}(\alpha(\omega,x))\mathrm{d}{\mu}(\omega,x).
  \]
  The \emph{conditional information function} of $\alpha$ with respect to $\eta$ is defined as
  \[
    I_{{\mu}}(\alpha|\eta)(\omega,x):=-\log{\mu}_{(\omega,x)}^\eta(\alpha(\omega,x)),
  \]
 where $\{{\mu}^\eta_{(\omega,x)}\}_{(\omega,x)\in\Omega\times M}$ is a canonical system of conditional measures  of  ${\mu}$ with respect to $\eta$. Then \emph{the conditional entropy of $\alpha$ with respect to $\eta$} is defined as
   \[
    H_{{\mu}}(\alpha|\eta):=\int_{\Omega\times M}I_{{\mu}}(\alpha|\eta)(\omega,x)\mathrm{d}{\mu}(\omega,x)=-\int_{\Omega\times M}\log{\mu}_{(\omega,x)}^\eta(\alpha(\omega,x))\mathrm{d}{\mu}(\omega,x).
  \]
\end{definition}

For simplicity, sometimes we will use the following notations. For $\omega\in\Omega$, denote
\[
  H_{{\mu}_\omega}(\alpha):=\int_M I_{{\mu}}(\alpha)(\omega,x)\mathrm{d}{\mu}_\omega(x),
\]
and
\[
  H_{{\mu}_\omega}(\alpha|\eta):=\int_M I_{{\mu}}(\alpha|\eta)(\omega,x)\mathrm{d}{\mu}_\omega(x).
\]
It is clear that
\[
  H_{{\mu}}(\alpha)=\int_\Omega H_{{\mu}_\omega}(\alpha)\mathrm{d}\mathbf{P}(\omega),
\]
and
\[
  H_{{\mu}}(\alpha|\eta)=\int_\Omega H_{{\mu}_\omega}(\alpha|\eta)\mathrm{d}\mathbf{P}(\omega).
\]

The following lemmas are derived from \cite{HuHuaWu2017} with slight adaption, which are useful for the proofs of our main results.
\begin{lemma}\label{lem:info}
  Given $\mu\in\mathcal{M}_{\mathbf{P}}(\mathcal{F})$ and let $\alpha$, $\beta$, and $\gamma$ be measurable partitions of $\Omega\times M$ with $H_{{\mu}}(\alpha|\gamma)$, $H_{{\mu}}(\beta|\gamma)<\infty$.
  \begin{enumerate}[label=(\roman*)]
    \item If $\alpha\leq\beta$, then $I_{{\mu}}(\alpha|\gamma)(\omega,x)\leq I_{{\mu}}(\beta|\gamma)(\omega,x)$ and $H_{{\mu}}(\alpha|\gamma)\leq H_{{\mu}}(\beta|\gamma)$;
    \item$I_{{\mu}}(\alpha\vee\beta|\gamma)(\omega,x)=I_{{\mu}}(\alpha|\gamma)(\omega,x)+I_{{\mu}}(\beta|\alpha\vee\gamma)(\omega,x)
    $
    and $H_{{\mu}}(\alpha\vee\beta|\gamma)=H_{{\mu}}(\alpha|\gamma)+H_{{\mu}}(\beta|\alpha\vee\gamma)$;
    \item $H_{{\mu}}(\alpha\vee\beta|\gamma)\leq H_{{\mu}}(\alpha|\gamma)+H_{{\mu}}(\beta|\gamma);$
    \item if $\beta\leq\gamma$, then $H_{{\mu}}(\alpha|\beta)\geq H_{{\mu}}(\alpha|\gamma)$.
  \end{enumerate}
\end{lemma}
\begin{lemma}\label{lem:info2}
  Given $\mu\in\mathcal{M}_{\mathbf{P}}(\mathcal{F})$. Let $\alpha$, $\beta$ and $\gamma$ be measurable partitions of $\Omega\times M$.
  \begin{enumerate}[label=(\roman*)]
    \item\label{lem:info2_item1}
    \[
      I_{{\mu}}(\beta_0^{n-1}|\gamma)(\omega,x)=I_{{\mu}}(\beta|\gamma)(\omega,x)+\sum_{i=1}^{n-1}I_{{\mu}}(\beta|\Theta^i(\beta_0^{i-1}\vee\gamma))(\Theta^i(\omega,x)),
    \]
    hence
    \[
      H_{{\mu}}(\beta_0^{n-1}|\gamma)=H_{{\mu}}(\beta|\gamma)+\sum_{i=1}^{n-1}H_{{\mu}}(\beta|\Theta^i(\beta_0^{i-1}\vee\gamma));
    \]
    \item\label{lem:info2_item2}
    \begin{align*}
       & I_{{\mu}}(\alpha_0^{n-1}|\gamma)(\omega,x) \\
      =& I_{{\mu}}(\alpha|\Theta^{n-1}\gamma)(\Theta^{n-1}(\omega,x))+\sum_{i=0}^{n-2}I_{{\mu}}(\alpha|\alpha_1^{n-1-i}\vee \Theta^i\gamma)(\Theta^i(\omega,x)),
    \end{align*}
    hence
    \[
      H_{{\mu}}(\alpha_0^{n-1}|\gamma)=H_{{\mu}}(\alpha|\Theta^{n-1}\gamma)+\sum_{i=0}^{n-2}H_{{\mu}}(\alpha|\alpha_1^{n-1-i}\vee \Theta^i\gamma).
    \]
  \end{enumerate}
\end{lemma}
\begin{lemma}\label{lem:semicontiofparti}
  Let $\alpha\in\mathcal{P}(\Omega\times M)$ and $\{\zeta_n\}$ be a sequence of increasing measurable partitions of $\Omega\times M$ with $\zeta_n\nearrow\zeta$. Then for $\varphi_n(\omega,x)=I_{{\mu}}(\alpha|\zeta_n)(\omega,x)$, $\varphi^*:=\sup_n\varphi_n\in L^1({\mu})$.
\end{lemma}
\begin{lemma}\label{lem:contiofparti}
  Let $\alpha\in\mathcal{P}(\Omega\times M)$ and $\{\zeta_n\}$ be a sequence of increasing measurable {partitions of $\Omega\times M$} with $\zeta_n\nearrow\zeta$. Then
  \begin{enumerate}[label=(\roman*)]
    \item $\lim_{n\to\infty}I_{{\mu}}(\alpha|\zeta_n)(\omega,x)=I_{{\mu}}(\alpha|\zeta)(\omega,x)$ for ${\mu}$-a.e. $(\omega,x)$, and
    \item\label{lem:contiofpartiitem2} $\lim_{n\to\infty}H_{{\mu}}(\alpha|\zeta_n)=H_{{\mu}}(\alpha|\zeta)$.
  \end{enumerate}
\end{lemma}

\section{Unstable Metric Entropy}\label{sec:Meaentropy}

Given $\mu\in\mathcal{M}_{\mathbf{P}}(\mathcal{F})$. In this section we give the definition of unstable metric entropy along $W^u$-foliation. Firstly, we give the definition using Bowen balls. {For any $(\omega,x)\in \Omega\times M$, we always let $d^{u}$ denote the distance induced by the Riemannian structure of $W^u(\omega,x)$. Let $V^u(\mathcal{F},\omega,x,n,\epsilon)$ denote the $d^u_{\omega,n}$-Bowen ball in $W^u(\omega,x)$ with center $x$ and radius $\epsilon$, i.e.,
\[
  V^u(\mathcal{F},\omega,x,n,\epsilon):=\{y\in W^u(\omega,x)\colon d^u_{\omega,n}(x,y)<\epsilon\},
\]
where
\[
   d^u_{\omega,n}(x,y):=\max_{0\leq j\leq n-1}\{ d^{u}(f_\omega^j(x),f^j_\omega(y))\}.
\]}

\begin{definition}\label{def:metricentropy1}
Let $\xi_u\in\mathcal{Q}^u(\Omega\times M)$, i.e., $\xi_u$ is an increasing partition of $\Omega\times M$ subordinate to $W^u$-foliation as in Proposition \ref{prop:specialpartition}. Now we define the \emph{the unstable metric entropy} along $W^u$-foliation as follows,
\[
  h_\mu(\mathcal{F},\xi_u)=\int_\Omega\int_M\lim_{\epsilon\to 0}\limsup_{n\to\infty}-\frac{1}{n}\log{\mu}_{(\omega,x)}^{\xi_u}V^u(\mathcal{F},\omega,x,n,\epsilon)\mathrm{d}{\mu}_\omega(x)\mathrm{d}\mathbf{P}(\omega).
\]

\end{definition}

We denote
\[
  \displaystyle{\lim_{\epsilon\to 0}\limsup_{n\to\infty}-\frac{1}{n}\log{\mu}_{(\omega,x)}^{\xi_u}V(\mathcal{F},\omega,x,n,\epsilon)}
\]
by $h_\mu(\mathcal{F},\omega,x,\xi_u)$, and denote
\[
  \int_Mh(\mathcal{F},\omega,x,\xi_u)\mathrm{d}{\mu}_\omega(x)
\]
by $h_\mu(\mathcal{F},\omega,\xi_u)$. It has been proved in \cite{QianQianXie2003} that $h_\mu(\mathcal{F},\omega,x,\xi_u)$ is $\Theta$-invariant and it is independent of the choice of $\xi_u$. Hence we also denote $\tilde{h}^u_\mu(\mathcal{F})= h_\mu(\mathcal{F},\xi_u)$.

\begin{remark}
  In fact, by the following lemma, we can replace `` $\limsup$'' by `` $\lim$'' {and remove ``$\lim_{\epsilon\to0}$''} in Definition \ref{def:metricentropy1}.
\end{remark}
\begin{lemma}\label{lem:low=up}
  Let $\xi_u\in\mathcal{Q}^u(\Omega\times M)$, then we have
  \[
    h_\mu(\mathcal{F},\omega,x,\xi_u)=\lim_{n\to\infty}-\frac{1}{n}\log{\mu}_{(\omega,x)}^{\xi_u}V^u(\mathcal{F},\omega,x,n,\epsilon),\quad{\mu}\text{-a.e.}(\omega,x).
  \]
\end{lemma}
\begin{pf}
  Denote
  \begin{equation*}\label{eq:lower}
    \underline{h}(\mathcal{F},\omega,x,\epsilon,\xi_u)=\liminf_{n\to\infty}-\frac{1}{n}\log{\mu}_{(\omega,x)}^{\xi_u}V^u(\mathcal{F},\omega,x,n,\epsilon)
  \end{equation*}
  and
  \begin{equation*}\label{eq:upper}
    \overline{h}(\mathcal{F},\omega,x,\epsilon,\xi_u)=\limsup_{n\to\infty}-\frac{1}{n}\log{\mu}_{(\omega,x)}^{\xi_u}V^u(\mathcal{F},\omega,x,n,\epsilon).
  \end{equation*}
  It has been proved in \cite{QianQianXie2003} that
  \begin{equation*}\label{eq:low=up}
    \lim_{\epsilon\to0}\underline{h}(\mathcal{F},\omega,x,\epsilon,\xi_u)=\lim_{\epsilon\to0}\overline{h}(\mathcal{F},\omega,x,\epsilon,\xi_u).
  \end{equation*}
  Thus we only need to prove both $\overline{h}(\mathcal{F},\omega,x,\epsilon,\xi_u)$ and $\underline{h}(\mathcal{F},\omega,x,\epsilon,\xi_u)$ are independent of $\epsilon$. By Definition \ref{domination}, we know that $\mathcal{F}$ is uniformly expanding restricted to $W^u$-foliation, so for any $0<\delta<\epsilon$, there exists $k>0$ such that
  \[
    V^u(\mathcal{F},\omega,x,k,\epsilon)\subset B^u(\omega,x,\delta),
  \]
  for any $x\in M$, where $B^u(\omega,x,\delta)$ is the ball in $W^u(\omega,x)$ centered at $x$ with radius $\delta$. Hence, for all $n>0$, we have
  \[
    V^u(\mathcal{F},\omega,x,n+k,\epsilon)\subset V^u(\mathcal{F},\omega,x,n,\delta)\subset V^u(\mathcal{F},\omega,x,n,\epsilon).
  \]
  Thus, we get
  \[
    \liminf_{n\to\infty}-\frac{1}{n}\log{\mu}_{(\omega,x)}^{\xi_u}V^u(\mathcal{F},\omega,x,n,\delta)=\liminf_{n\to\infty}-\frac{1}{n}\log{\mu}_{(\omega,x)}^{\xi_u}V^u(\mathcal{F},\omega,x,n,\epsilon),
  \]
  and
  \[
    \limsup_{n\to\infty}-\frac{1}{n}\log{\mu}_{(\omega,x)}^{\xi_u}V^u(\mathcal{F},\omega,x,n,\delta)=\limsup_{n\to\infty}-\frac{1}{n}\log{\mu}_{(\omega,x)}^{\xi_u}V^u(\mathcal{F},\omega,x,n,\epsilon),
  \]
  from which Lemma \ref{lem:low=up} follows.
\end{pf}

Now we consider the definition using {fiberwise finite} partitions.
\begin{definition}\label{def:metricentropy2}
  Given $\mu\in\mathcal{M}_{\mathbf{P}}(\mathcal{F})$. {The conditional entropy of $\mathcal{F}$ for a {fiberwise finite} measurable partition $\alpha$ with respect to} $\eta\in\mathcal{P}^u(\Omega\times M)$ is defined as
  \[
    h_\mu(\mathcal{F},\alpha|\eta)=\limsup_{n\to\infty}\frac{1}{n}H_{{\mu}}(\alpha^{n-1}_0|\eta).
  \]
  {The conditional entropy of $\mathcal{F}$ with respect to $\eta$} is defined as
  \[
    h_{\mu}(\mathcal{F}|\eta)=\sup_{\alpha\in\mathcal{P}(\Omega\times M)}h_\mu(\mathcal{F},\alpha|\eta),
  \]
  and the \emph{conditional entropy} of $\mathcal{F}$ along $W^u$-foliation is defined as
  \[
    {h}^u_\mu(\mathcal{F})=\sup_{\eta\in\mathcal{P}^u(\Omega\times M)}h_{\mu}(\mathcal{F}|\eta).
  \]
\end{definition}

In the following, we give some details on the increasing partitions in $\mathcal{Q}^u(\Omega\times M)$. Denote by $\tilde{d}(\cdot,\cdot)$ the metric on $\Omega$ by which the $\sigma$-algebra $\mathscr{F}$ can be induced. Suppose $\mu\in\mathcal{M}_{\mathbf{P}}(\mathcal{F})$ is ergodic. By Liu and Qian's argument in \cite{LiuQian1995} and Bahnm\"uller and Liu's argument in \cite{BahnmullerLiu1998}, we can choose a set $\tilde{\Lambda}\subset\Lambda$, $(\omega_0,x_0)\in\tilde{\Lambda}$ and positive constants $\hat{\epsilon}$, $\hat{r}$ with
\[
  B_{\tilde{\Lambda}}:=B_{\tilde{\Lambda}}((\omega_0,x_0),\hat{\epsilon}\hat{r}/2)=\{(\omega,x)\in\Omega\times M\colon \tilde{d}(\omega,\omega_0)<\hat{\epsilon}\hat{r}/2\text{, }d(x,x_0)<\hat{\epsilon}\hat{r}/2\}
\]
having positive ${\mu}$ measure such that the following construction of a partition $\xi_u$ satisfies Proposition \ref{prop:specialpartition}.

For each $r\in[\hat{r}/2,\hat{r}]$, put
\[
  S_{u,r}=\bigcup_{(\omega,x)\in B_{\tilde{\Lambda}}}S_u(\omega,x,r),
\]
where $S_u(\omega,x,r)=\{\omega\}\times(W^u_{\text{loc}}(\omega,x)\cap B(x_0,r))$. Then we can define a partition $\hat{\xi}_{u,x_0}$ of $\Omega\times M$ such that
\[
  (\hat{\xi}_{u,x_0})(\omega,y)=
  \begin{cases}
    S_u(\omega,x,r), & {y\in W^u_{\text{loc}}(\omega,x)\cap B(x_0,r) \text{\ for some\ }(\omega,x)\in B_{\tilde{\Lambda}},} \\
    (\Omega\times M)\setminus S_{u,r}, & \mbox{otherwise}.
  \end{cases}
\]
Next we can choose an appropriate $r\in[\hat{r}/2,\hat{r}]$ such that
\[
  \xi_u=\bigvee_{j=0}^\infty\Theta^j\hat{\xi}_{u,x_0}
\]
is subordinate to $W^u$-foliation. We will use the notation $\hat{\xi}_{u,-k}=\bigvee_{j=0}^k\Theta^j\hat{\xi}_{u,x_0}$.

Given $\mu\in\mathcal{M}_{\mathbf{P}}(\mathcal{F})$. Proposition \ref{prop:localvsfinite2} below shows that the two definitions above using Bowen balls and {fiberwise finite} partitions are equivalent when $\mu\in\mathcal{M}_{\mathbf{P}}(\mathcal{F})$ is ergodic. Firstly, we need some lemmas.
\begin{lemma}\label{lem:approach}
  Suppose $\mu$ is an ergodic measure and $\alpha\in\mathcal{P}(\Omega\times M)$ is fiberwise finite. For any $\epsilon>0$, there exists $K>0$ such that for any $k\geq K$,
  \[
    \limsup_{n\to\infty}H_{{\mu}}(\alpha|\alpha^n_1\vee(\hat{\xi}_{u,-k})_1^n)\leq\epsilon.
  \]
\end{lemma}
\begin{pf}
  Denote $S_{-k}=\bigcup_{j=0}^k\Theta^jS_{u,r}$. Let $\epsilon>0$. Since ${\mu}$ is ergodic, we have ${\mu}(S_{-k})\to1$ as $k\to\infty$. {Recall that $\alpha$ is fiberwise finite, hence $\int_{\Omega}K(\omega)\mathrm{d}\mathbf{P}(\omega)<\infty$ where $K(\omega)$ is the cardinality of $\alpha_\omega$. It follows that $\int_{\Omega\times M}K(\omega)\mathrm{d}{\mu}(\omega,x)<\infty$. Thus there exists $K>0$ such that for any $k\geq K$, we have
\[
  \int_{(\Omega\times M)\setminus S_{-k}}K(\omega)\mathrm{d}{\mu}(\omega,x)<\epsilon.
\]
} Then
  \begin{align*}
     & H_{{\mu}}(\alpha|\alpha_1^n\vee(\hat{\xi}_{u,-k})_1^n) \\
    =& \int_{S_{-k}}I_{{\mu}}(\alpha|\alpha_1^n\vee(\hat{\xi}_{u,-k})_1^n)\mathrm{d}{\mu}(\omega,x)+\int_{(\Omega\times M)\setminus S_{-k}}I_{{\mu}}(\alpha|\alpha_1^n\vee(\hat{\xi}_{u,-k})_1^n)\mathrm{d}{\mu}(\omega,x).
  \end{align*}
  For $(\omega,x)\in S_{-k}$, $\alpha_1^n\vee(\hat{\xi}_{u,-k})_1^n(\omega,x)\subset W^u_{\text{loc}}(\omega,x)$. Hence for almost every $(\omega,x)\in S_{-k}$, there exists $N=N(\omega,x)>0$ such that for any $n\geq N$, $\alpha_1^n\vee(\hat{\xi}_{u,-k})_1^n(\omega,x)\subset\alpha(\omega,x)$, which implies that $\log{\mu}_{(\omega,x)}^{\alpha_1^n\vee(\hat{\xi}_{u,-k})_1^n}(\alpha(\omega,x))=0$. {Lemma \ref{lem:semicontiofparti} with $\zeta_n=\alpha_1^n\vee(\hat{\xi}_{u,-k})_1^n$ and Lebesgue's Dominated Convergence Theorem imply that}
  \[
    \limsup_{n\to\infty}\int_{S_{-k}}I_{{\mu}}(\alpha|\alpha_1^n\vee(\hat{\xi}_{u,-k})_1^n)\mathrm{d}{\mu}(\omega,x)=0.
  \]
  {For $(\omega,x)\in(\Omega\times M)\setminus S_{-k}$, we have
  \[
    \int_{\alpha_1^n\vee(\hat{\xi}_{u,-k})_1^n(\omega,x)}-\log{\mu}_{(\omega,x)}^{\alpha_1^n\vee(\hat{\xi}_{u,-k})_1^n}(\alpha(\omega,y))\mathrm{d}{\mu}_{(\omega,x)}^{\alpha_1^n\vee(\hat{\xi}_{u,-k})_1^n}(\omega,y)\leq K(\omega),
  \]
  which implies that
 \begin{align*}
    &\int_{(\Omega\times M)\setminus S_{-k}}-\log{\mu}_{(\omega,x)}^{\alpha_1^n\vee(\hat{\xi}_{u,-k})_1^n}(\alpha(\omega,x))\mathrm{d}{\mu}(\omega,x)\\
    \leq &\int_{(\Omega\times M)\setminus S_{-k}}K(\omega)\mathrm{d}{\mu}(\omega,x)\leq \epsilon.
  \end{align*}}
  Thus we get what we need.
\end{pf}
\begin{lemma}\label{lem:convergence}
  Let $\mu$ be an ergodic measure. Suppose $\eta\in\mathcal{P}^u(\Omega\times M)$ is subordinate to $W^u$-foliation, and $\hat{\xi}_{u,-k}$ is a partition described as above, where $k\in\mathbb{N}\cup\{\infty\}$. Then for almost every $(\omega,x)$, there exists $N=N(\omega,x)>0$ such that for any $j>N$, we have
  \[
    (\hat{\xi}_{u,-k-j}\vee\Theta^j\eta)(\Theta^j(\omega,x))=(\hat{\xi}_{u,-k-j})(\Theta^j(\omega,x)).
  \]
  Hence, for any partition $\beta$ of $\Omega\times M$ with $H_{{\mu}}(\beta|\hat{\xi}_{u,-k})<\infty$,
  \[
    I_{{\mu}}(\beta|\hat{\xi}_{u,-k-j}\vee\Theta^j\eta)(\Theta^j(\omega,x))=I_{{\mu}}(\beta|\hat{\xi}_{u,-k-j})(\Theta^j(\omega,x)),
  \]
  which implies that
  \[
    \lim_{j\to\infty}H_{{\mu}}(\beta|\hat{\xi}_{u,-k-j}\vee\Theta^j\eta)=H_{{\mu}}(\beta|\xi_u).
  \]
  Particularly, if we take $k=\infty$, then the last two equalities become
 \[
    I_{{\mu}}(\beta|\xi_u\vee\Theta^j\eta)(\Theta^j(\omega,x))=I_{{\mu}}(\beta|\xi_u)(\Theta^j(\omega,x)),
  \]
  and
  \[
    \lim_{j\to\infty}H_{{\mu}}(\beta|\xi_u\vee\Theta^j\eta)=H_{{\mu}}(\beta|\xi_u).
  \]
\end{lemma}
\begin{pf}
  Since $\eta$ is subordinate to $W^u$, for ${\mu}$-a.e. $(\omega,x)$, there exists $\rho=\rho(\omega,x)>0$ such that $B^u(\omega,x,\rho)\subset\eta_\omega(x)$. Since ${\mu}$ is ergodic, for ${\mu}$-a.e. $(\omega,x)$, there are infinitely many $n>0$ such that $\Theta^n(\omega,x)\in S_{u,r}$. Take $N=N(\omega,x)$ large enough such that
  \[
    \Theta^N(\omega,x)\in S_{u,r}
  \]
  and
  \[
   { f^{-N}_{\theta^N\omega}}((\hat{\xi}_{u,x_0})_{\Theta^N\omega}(f^{N}_\omega(x)))\subset B^u(\omega,x,\rho)\subset\eta_\omega(x).
  \]
  Then we have
  \[
    {f^{-j}_{\theta^j\omega}}\Big((f^{j-N}_{\theta^N\omega}(\hat{\xi}_{u,x_0})_{\theta^N\omega})(f^j_\omega(x))\Big)\subset\eta_\omega(x)
  \]
  for any $j\geq N$. Since
  \[
    (\hat{\xi}_{u,-k-j})_{\theta^j\omega}=\bigvee_{l=0}^{k+j}f^l_{\theta^{j-l}\omega}(\hat{\xi}_{u,x_0})_{\theta^{j-l}\omega}\geq f^{j-N}_{\theta^N\omega}(\hat{\xi}_{u,x_0})_{\theta^N\omega},
  \]
  so we have
  \[
    {f^{-j}_{\theta^j\omega}}\Big((\hat{\xi}_{u,-k-j})_{\theta^j\omega}(f^j_\omega(x))\Big)\subset\eta_\omega(x).
  \]
  Thus we have
  \[
    (\hat{\xi}_{u,-k-j})_{\theta^j\omega}(f^j_\omega(x))\subset(f^j_\omega\eta_\omega)(f^j_\omega(x)),
  \]
  which implies that
  \[
    \left((\hat{\xi}_{u,-k-j})_{\theta^j\omega}\vee f^j_\omega\eta_\omega\right)(f^j_\omega(x))=(\hat{\xi}_{u,-k-j})_{\theta^j\omega}(f^j_\omega(x)).
  \]
  We get the first equality.

  By the definition of information function, it is clearly that
  \[
    I_{{\mu}}(\beta|\hat{\xi}_{u,-k-j}\vee\Theta^j\eta)(\Theta^j(\omega,x))=I_{{\mu}}(\beta|\hat{\xi}_{u,-k-j})(\Theta^j(\omega,x)).
  \]
  Now we get
  \[
    I_{{\mu}}(\beta|\hat{\xi}_{u,-k-j}\vee\Theta^j\eta)(\Theta^j(\omega,x))-I_{{\mu}}(\beta|\hat{\xi}_{u,-k-j})(\Theta^j(\omega,x))=0
  \]
    for ${\mu}$-a.e. $(\omega,x)$.
  Let
  \[
    \varphi_j=(I_{{\mu}}(\beta|\hat{\xi}_{u,-k-j}\vee\Theta^j\eta)-I_{{\mu}}(\beta|\hat{\xi}_{u,-k-j}))\circ\Theta^j.
  \]
  Then
  \[
    \lim_{j\to\infty}\varphi_j(\omega,x)=0
  \]
  for ${\mu}$-a.e. $(\omega,x)$. By Fatou's Lemma, we have
  \[
    \liminf_{j\to\infty}\int\varphi_j\mathrm{d}\mu\geq\int\liminf_{j\to\infty}\varphi_j\mathrm{d}\mu=0,
  \]
  which means that
  \[
    \liminf_{j\to\infty}H_{{\mu}}(\beta|\hat{\xi}_{u,-k-j}\vee\Theta^j\eta)\geq \lim_{j\to\infty}H_{{\mu}}(\beta|\hat{\xi}_{u,-k-j}).
  \]
  By Lemma \ref{lem:contiofparti} \ref{lem:contiofpartiitem2} with $\zeta_j=\hat{\xi}_{u,-k-j}$ and $\zeta=\xi_u$, we have
  \[
    \lim_{j\to\infty}H_{{\mu}}(\beta|\hat{\xi}_{u,-k-j})=H_{{\mu}}(\beta|\xi_u).
  \]
  It is clear that $H_{{\mu}}(\beta|\hat{\xi}_{u,-k-j}\vee\Theta^j\eta)\leq H_{{\mu}}(\beta|\hat{\xi}_{u,-k-j})$ for any $j>0$. It follows that
  \[
    \limsup_{j\to\infty}H_{{\mu}}(\beta|\hat{\xi}_{u,-k-j}\vee\Theta^j\eta)\leq \lim_{j\to\infty}H_{{\mu}}(\beta|\hat{\xi}_{u,-k-j}).
  \]
  Now we get the last equality.
\end{pf}
\begin{lemma}(Cf. \cite{QianQianXie2003} and \cite{LiuXie2006})\label{lem:locvsfinite}
 Let $\mu$ be ergodic and $\xi_u\in\mathcal{Q}^u(\Omega\times M)$. Then for ${\mu}$-a.e. $(\omega,x)\in\Omega\times M$,
  \[
    h_\mu(\mathcal{F},\omega,x,\xi_u)=H_{{\mu}}(\xi_u|\Theta\xi_u).
  \]
\end{lemma}
\begin{proposition}\label{prop:localvsfinite2}
  Suppose $\mu$ is ergodic. Let $\xi_u\in\mathcal{Q}^u(\Omega\times M)$, $\alpha\in\mathcal{P}(\Omega\times M)$ and $\eta\in\mathcal{P}^u(\Omega\times M)$, then
  \[
    h_\mu(\mathcal{F},\omega,x,\xi_u)=\lim_{n\to\infty}\frac{1}{n}H_{{\mu}}(\alpha_0^{n-1}|\eta),
  \]
  for ${\mu}$-a.e. $(\omega,x)\in\Omega\times M$.
\end{proposition}
\begin{pf}
  The proof is similar to which in \cite{HuHuaWu2017}. Firstly, let us show that $h_{\mu}(\mathcal{F},\alpha|\eta)$ is independent of $\eta$, i.e., for $\eta_1$ and $\eta_2\in\mathcal{P}^u(\Omega\times M)$, we have
  \[
    h_{\mu}(\mathcal{F},\alpha|\eta_1)=h_{\mu}(\mathcal{F},\alpha|\eta_2).
  \]
  In fact, by Lemma \ref{lem:info}, we have
  \begin{align}\label{eq:compute}
    H_{{\mu}}(\alpha_0^{n-1}|\eta_1)+H_{{\mu}}(\eta_2|\alpha_0^{n-1}\vee\eta_1)=& H_{{\mu}}(\alpha_0^{n-1}|\eta_2\vee\eta_1)+H_{{\mu}}(\eta_2|\eta_1), \notag\\
    H_{{\mu}}(\alpha_0^{n-1}|\eta_2)+H_{{\mu}}(\eta_1|\alpha_0^{n-1}\vee\eta_2)=& H_{{\mu}}(\alpha_0^{n-1}|\eta_1\vee\eta_2)+H_{{\mu}}(\eta_1|\eta_2).
  \end{align}
  Because of the construction of $\eta_1$ and $\eta_2$, there exist two fiberwise finite partitions $\alpha_1$ and $\alpha_2$ such that $\eta_j(\omega,x)=\alpha_j(\omega,x)\cap W^u_{\text{loc}}(\omega,x)$, $j=1,2$, for all $(\omega,x)\in \Omega\times M$. Let $N_1(\omega)$ and $N_2(\omega)$ be the cardinality of ${\alpha_1}_\omega$ and ${\alpha_2}_\omega$ respectively. For any $(\omega,x)\in\Omega\times M$, $\eta_1(\omega,x)$ intersects at most $N_2(\omega)$ elements of ${\alpha_2}_\omega$, so intersects at most $N_2(\omega)$ elements of $\eta_2$. Thus, we have
  \[
    \lim_{n\to\infty}\frac{1}{n}H_{{\mu}}(\eta_2|\alpha_0^{n-1}\vee\eta_1)\leq\lim_{n\to\infty}\frac{1}{n}H_{{\mu}}(\eta_2|\eta_1)\leq\lim_{n\to\infty}\frac{1}{n}\int_{\Omega}N_2(\omega)\mathrm{d}\mathbf{P}(\omega)=0.
  \]
  Similarly, we have
  \[
    \lim_{n\to\infty}\frac{1}{n}H_{{\mu}}(\eta_1|\alpha_0^{n-1}\vee\eta_2)\leq\lim_{n\to\infty}\frac{1}{n}H_{{\mu}}(\eta_1|\eta_2)=0.
  \]
  Hence we by \eqref{eq:compute}, we get
  \[
    \limsup_{n\to\infty}\frac{1}{n}H_{{\mu}}(\alpha_0^{n-1}|\eta_1)=\limsup_{n\to\infty}\frac{1}{n}H_{{\mu}}(\alpha_0^{n-1}|\eta_2).
  \]
  Thus $h_{\mu}(\mathcal{F},\alpha|\eta)$ is independent of $\eta$.

  Then we show that $h_{\mu}(\mathcal{F},\alpha|\eta)$ is independent of $\alpha$, that is, for any $\beta$, $\gamma\in\mathcal{P}(\Omega\times M)$,
  \[
    \limsup_{n\to\infty}\frac{1}{n}H_{{\mu}}(\beta_0^{n-1}|\eta)=\limsup_{n\to\infty}\frac{1}{n}H_{{\mu}}(\gamma_0^{n-1}|\eta).
  \]
  In fact, by Lemma \ref{lem:info}, we have
  \begin{equation}\label{eq:betagamma1}
    H_{{\mu}}(\beta_0^{n-1}|\eta)\leq H_{{\mu}}(\gamma_0^{n-1}|\eta)+H_{{\mu}}(\beta_0^{n-1}|\gamma_0^{n-1}\vee\eta),
  \end{equation}
  and similar to the proof of Lemma 2.7 (\romannumeral2) in \cite{HuHuaWu2017}, we can show that
  \begin{equation}\label{eq:betagamma2}
    \lim_{n\to\infty}\frac{1}{n}H_{{\mu}}(\beta_0^{n-1}|\gamma_0^{n-1}\vee\eta)=0.
  \end{equation}
  By \eqref{eq:betagamma1} and \eqref{eq:betagamma2}, we have
  \[
    \limsup_{n\to\infty}\frac{1}{n}H_{{\mu}}(\beta_0^{n-1}|\eta)\leq \limsup_{n\to\infty}\frac{1}{n}H_{{\mu}}(\gamma_0^{n-1}|\eta).
  \]
  By the arbitrariness of $\beta$ and $\gamma$, we obtain
  \[
    \limsup_{n\to\infty}\frac{1}{n}H_{{\mu}}(\beta_0^{n-1}|\eta)=\limsup_{n\to\infty}\frac{1}{n}H_{{\mu}}(\gamma_0^{n-1}|\eta).
  \]

  Now we start the proof of the proposition. By Lemma \ref{lem:info2} \ref{lem:info2_item1}, let $\gamma=\eta$ and $\beta=\hat{\xi}_{u,-k}$, we have that for any $\eta\in\mathcal{P}^u(\Omega\times M)$, $n>0$,
  \[
    \frac{1}{n}H_{{\mu}}((\hat{\xi}_{u,-k})_0^{n-1}|\eta)=\frac{1}{n}H_{{\mu}}(\hat{\xi}_{u,-k}|\eta)+\frac{1}{n}\sum_{j=0}^{n-1}H_{{\mu}}(\hat{\xi}_{u,-k}|\Theta\hat{\xi}_{u,-k-j+1}\vee\Theta^j\eta).
  \]
  By Lemma \ref{lem:convergence}, the right side of above equality converges to $H_{{\mu}}(\hat{\xi}_{u,-k}|\Theta\xi_u)$ as $j\to\infty$. It is clear that each elements of $\eta_\omega$ intersects at most {$2^{k+1}$} elements of $(\hat{\xi}_{u,-k})_\omega$. So we have
  \[
    H_{{\mu}}(\hat{\xi}_{u,-k}|\eta)=\int_{\Omega}H_{{{\mu}}_\omega}(\hat{\xi}_{u,-k}|\eta)\mathrm{d}\mathbf{P}(\omega)\leq\int_\Omega\log{2^{k+1}}\mathrm{d}\mathbf{P}(\omega)=\log{2^{k+1}}.
  \]
  Hence we have
  \[
    \lim_{n\to\infty}\frac{1}{n}H_{{\mu}}(\hat{\xi}_{u,-k}|\eta)=0.
  \]
  Thus we get
  \begin{equation}\label{eq:finitevsmeas}
    \lim_{n\to\infty}\frac{1}{n}H_{{\mu}}((\hat{\xi}_{u,-k})_0^{n-1}|\eta)=H_{{\mu}}(\hat{\xi}_{u,-k}|\Theta\xi_u)\leq H_{{\mu}}(\xi_u|\Theta\xi_u).
  \end{equation}
  Replacing $\gamma$ by $(\hat{\xi}_{u,-k})_0^{n-1}$ in Lemma \ref{lem:info2} \ref{lem:info2_item2} and noticing that
  \[
    \Theta^j(\hat{\xi}_{u,-k})_0^{n-1}=(\hat{\xi}_{u,-k-j})_0^{n-j-1},
  \]
  we have
  \begin{align*}
    H_{{\mu}}(\alpha_0^{n-1}|(\hat{\xi}_{u,-k})_0^{n-1}) &=H_{{\mu}}(\alpha|\hat{\xi}_{u,-n-k+1})+\sum_{j=0}^{n-2}H_{{\mu}}(\alpha|\alpha_1^{n-1-j}\vee(\hat{\xi}_{u,-k-j})_0^{n-1-j}) \\
    &=H_{{\mu}}(\alpha|\hat{\xi}_{u,-n-k+1})+\sum_{j=1}^{n-1}H_{{\mu}}(\alpha|\alpha_1^{j}\vee{\hat{\xi}^j}_{u,-k-n+1+j}) \\
    &\leq H_{{\mu}}(\alpha)+\sum_{j=1}^{n-1}H_{{\mu}}(\alpha|\alpha_1^j\vee(\hat{\xi}_{u,-k})_1^j).
  \end{align*}
  For any $\epsilon>0$, take $k>0$ as in Lemma \ref{lem:approach}, thus we have
  \[
    \limsup_{n\to\infty}H_{{\mu}}(\alpha|\alpha_1^{n-1}\vee(\hat{\xi}_{u,-k})_1^{n-1})\leq \epsilon.
  \]
  Then we get
   \begin{equation}\label{eq:finitevsmeas2}
    \limsup_{n\to\infty}\frac{1}{n}H_{{\mu}}(\alpha_0^{n-1}|(\hat{\xi}_{u,-k})_0^{n-1})\leq\epsilon.
  \end{equation}
  By Lemma \ref{lem:info}, we have
  \begin{equation}\label{eq:finitevsmeas3}
    H_{{\mu}}(\alpha_0^{n-1}|\eta)\leq H_{{\mu}}((\hat{\xi}_{u,-k})_0^{n-1}|\eta)+H_{{\mu}}(\alpha_0^{n-1}|(\hat{\xi}_{u,-k})^{n-1}_0).
  \end{equation}

  Thus, by \eqref{eq:finitevsmeas2}, \eqref{eq:finitevsmeas3}, then by \eqref{eq:finitevsmeas} and Lemma \ref{lem:locvsfinite} we have
  \begin{align*}
    h_\mu(\mathcal{F},\alpha|\eta) &=\limsup_{n\to\infty}\frac{1}{n}H_{{\mu}}(\alpha_0^{n-1}|\eta) \\
    &\leq \lim_{n\to\infty}\frac{1}{n}H_{{\mu}}((\hat{\xi}_{u,-k})_0^{n-1}|\eta)+\epsilon \\
    &{\leq H_{{\mu}}(\xi_u|\Theta\xi_u)+\epsilon} \\
    &=h_\mu(\mathcal{F},\omega,x,\xi_u)+\epsilon.
  \end{align*}
  Since $\epsilon$ is arbitrary, we have $h_\mu(\mathcal{F},\alpha|\eta)\leq h_\mu(\mathcal{F},\omega,x,\xi_u).$

It remains to prove that $h_\mu(\mathcal{F},\alpha|\eta)\geq h_\mu(\mathcal{F},\omega,x,\xi_u).$  Let $\xi_u\in \mathcal{Q}^u(\Omega\times M)$. Since $\xi_u$ is a generator, we can choose $N$ large enough such that the measurable partition $\tilde{\xi}:=\bigvee_{j=0}^N\Theta^{-j}{\xi}_{u}$ has diameter small enough. It is clear that $\tilde{\xi}$ still satisfies the condition of Proposition \ref{prop:specialpartition}, and we know that $h_\mu(\mathcal{F},\tilde{\xi})=h_\mu(\mathcal{F},\xi_u)$.
  So we only need to prove the above inequality for $\tilde{\xi}$.

  We can find a sequence of partitions $\alpha_n\in\mathcal{P}(\Omega\times M)$ such that
  \[
    \mathscr{B}(\alpha_n)\nearrow\mathscr{B}(\Theta^{-1}\tilde{\xi})\text{ as }n\to\infty.
  \]
  So we have
  \[
    \lim_{n\to\infty}H_{{\mu}}(\alpha_n|\tilde{\xi})=H_{{\mu}}(\Theta^{-1}\tilde{\xi}|\tilde{\xi}).
  \]
  Thus,
  \[
    \sup_{\alpha\in\mathcal{P}(\Omega\times M),\alpha<\Theta^{-1}\tilde{\xi}}H_{{\mu}}(\alpha|\tilde{\xi})=H_{{\mu}}(\Theta^{-1}\tilde{\xi}|\tilde{\xi}).
  \]
  For any $\alpha\in\mathcal{P}(\Omega\times M)$ with $\alpha<\Theta^{-1}\tilde{\xi}$, we have that for any $j>0$, $\Theta^j\alpha^{j-1}_0<\Theta^j(\Theta^{-1}\tilde{\xi})^{j-1}_0=\tilde{\xi}$. Then by Lemma \ref{lem:info2} \ref{lem:info2_item1}, we have
  \begin{align*}
    H_{{\mu}}(\alpha_0^{n-1}|\eta) &= H_{{\mu}}(\alpha|\eta)+\sum_{j=1}^{n-1}H_{{\mu}}(\alpha|\Theta^j(\alpha_0^{j-1}\vee\eta)) \\
    &\geq H_{{\mu}}(\alpha|\eta)+\sum_{j=1}^{n-1}H_{{\mu}}(\alpha|\tilde{\xi}\vee\Theta^j\eta).
  \end{align*}
  Then by Lemma \ref{lem:convergence} we have
  \[
    \lim_{j\to\infty}H_{{\mu}}(\alpha|\tilde{\xi}\vee\Theta^j\eta)=H_{{\mu}}(\alpha|\tilde{\xi}),
  \]
  which implies that
  \[
    \limsup_{n\to\infty}\frac{1}{n}H_{{\mu}}(\alpha_0^{n-1}|\eta)\geq\liminf_{n\to\infty}\frac{1}{n}H_{{\mu}}(\alpha_0^{n-1}|\eta)\geq H_{{\mu}}(\alpha|\tilde{\xi}).
  \]
  So we have
  \begin{align*}
    \sup_{\alpha\in\mathcal{P}(\Omega\times M)}h_\mu(\mathcal{F},\alpha|\eta) &\geq\sup_{\alpha\in\mathcal{P}(\Omega\times M),\alpha<\Theta^{-1}\tilde{\xi}}h_\mu(\mathcal{F},\alpha|\eta) \\
    &= \sup_{\alpha\in\mathcal{P}(\Omega\times M),\alpha<\Theta^{-1}\tilde{\xi}}\limsup_{n\to\infty}\frac{1}{n}H_{{\mu}}(\alpha_0^{n-1}|\eta) \\
    &\geq \sup_{\alpha\in\mathcal{P}(\Omega\times M),\alpha<\Theta^{-1}\tilde{\xi}}\liminf_{n\to\infty}\frac{1}{n}H_{{\mu}}(\alpha_0^{n-1}|\eta) \\
    &\geq \sup_{\alpha\in\mathcal{P}(\Omega\times M),\alpha<\Theta^{-1}\tilde{\xi}}H_{{\mu}}(\alpha|\tilde{\xi}) \\
    &=H_{{\mu}}(\Theta^{-1}\tilde{\xi}|\tilde{\xi}).
  \end{align*}
  We have proved that $h_\mu(\mathcal{F},\alpha|\eta)$ is independent of $\alpha$, meaning
  \[
    h_\mu(\mathcal{F},\alpha|\eta)=\sup_{\beta\in\mathcal{P}(\Omega\times M)}h_\mu(\mathcal{F},\beta|\eta)
  \]
    for any $\alpha\in\mathcal{P}(\Omega\times M)$, which implies what we need.
\end{pf}

\begin{potA}
  This can be obtained directly from Proposition \ref{prop:localvsfinite2}.
\end{potA}

\begin{corollary}
  Suppose $\mu$ is ergodic, then for any $\alpha\in\mathcal{P}(\Omega\times M)$ and $\eta\in\mathcal{P}^u(\Omega\times M)$, we have
  \[
    {h}_\mu^u(\mathcal{F})=h_\mu(\mathcal{F},\alpha|\eta)=\lim_{n\to\infty}\frac{1}{n}H_{{\mu}}(\alpha^{n-1}_0|\eta).
  \]
\end{corollary}

\begin{pf}
 We can get the result directly from Proposition \ref{prop:localvsfinite2}.
\end{pf}

The following lemmas are useful for the proof of the variational principle in Section \ref{sec:VP}, whose proofs are completely parallel to those in \cite{HuHuaWu2017}, so we omit them.

\begin{lemma}\label{lem:affine}
  For any $\alpha\in\mathcal{P}(\Omega\times M)$ and $\eta\in\mathcal{P}^u(\Omega\times M)$, the map ${\mu}\mapsto H_{{\mu}}(\alpha|\eta)$ from $\mathcal{M}_{\mathbf{P}}(\mathcal{F})$ to $\mathbb{R}^+\cup \{0\}$ is concave. Moreover, the map $\mu\mapsto h_\mu^u(\mathcal{F})$ from $\mathcal{M}_{\mathbf{P}}(\mathcal{F})$ to $\mathbb{R}^+\cup \{0\}$ is affine.
\end{lemma}

\begin{lemma}\label{lem:semiconti}
  Let $\mu\in\mathcal{M}_{\mathbf{P}}(\mathcal{F})$ and {$\eta\in \mathcal{P}^u(\Omega\times M)$.} Assume that there exists a sequence of partitions $\{\beta_n\}_{n=1}^{\infty}\subset\mathcal{P}(\Omega\times M)$ such that $\beta_1<\beta_2<\cdots<\beta_n<\cdots$ and {$\mathscr{B}(\beta_n)\nearrow\mathscr{B}(\eta)$}, {and moreover, ${\mu}_\omega(\partial(\beta_n)_\omega)=0$, for $n=1,2,\cdots$ and $\mathbf{P}$-a.e. $\omega\in\Omega$. Let $\alpha\in\mathcal{P}(\Omega\times M)$ satisfy ${\mu}_\omega(\partial\alpha_\omega)=0$ for $\mathbf{P}$-a.e. $\omega\in\Omega$.} Then for $\mathbf{P}$-a.e. $\omega\in\Omega$, the function ${\mu}'\mapsto H_{{\mu}'}(\alpha|\eta)$ is upper semi-continuous at ${\mu}$, i.e.,
  \[
    \limsup_{{\mu}'\to{\mu}}H_{{\mu}'}(\alpha|\eta)\leq H_{{\mu}}(\alpha|\eta).
  \]
  Moreover, the function $\mu'\mapsto h_{\mu'}^u(\mathcal{F})$ is upper semi-continuous at $\mu$, i.e.,
  \[
    \limsup_{{\mu}'\to\mu}h_{\mu'}^u(\mathcal{F})\leq h_{\mu}^u(\mathcal{F}).
  \]
\end{lemma}

\section{Shannon-McMillan-Breiman Theorem}\label{sec:SMBThm}

In this section, we give a proof of Theorem \ref{thm:SMB}. We follow the method in \cite{HuHuaWu2017} to prove it, via which Hu, Hua and Wu give a version of Shannon-McMillan-Breiman Theorem for unstable metric entropy in deterministic case. Firstly, we need the following lemmas. {In this section, we always suppose that $\mu\in\mathcal{M}_{\mathbf{P}}(\mathcal{F})$ is ergodic.}

\begin{lemma}\label{lem:SMBlow}
  Let $\alpha\in\mathcal{P}(\Omega\times M)$, $\eta\in\mathcal{P}^u(\Omega\times M)$. Then for any $\xi\in\mathcal{Q}^u(\Omega\times M)$, we have
  \[
     {h_\mu(\mathcal{F},\alpha|\eta)}\leq\liminf_{n\to\infty}\frac{1}{n}I_{{\mu}}(\alpha_0^{n-1}|\xi)(\omega,x)\quad{\mu}\text{-a.e. }(\omega,x).
  \]
\end{lemma}

\begin{pf}
  For each $\omega\in\Omega$, take $k=k(\omega)$ such that $\mathrm{diam}(\alpha^k_0\vee\xi)_\omega\leq\epsilon$. Then for $n>0$, we have
  \[
    (\alpha_0^{k+n-1}\vee\xi)_\omega(x)=\bigvee_{j=0}^{n-1}( {\Theta^{-j}}\alpha^k_0\vee\xi)_\omega(x)\subset V^u(\mathcal{F},\omega,x,n,\epsilon).
  \]
  By Proposition \ref{prop:localvsfinite2} and Lemma \ref{lem:low=up} we know that
  \begin{align*}
    h_\mu(\mathcal{F},\alpha|\eta)&=h_\mu(\mathcal{F},\omega,x,\xi)\\
   & = \lim_{n\to\infty}-\frac{1}{n}\log{\mu}_{(\omega,x)}^{\xi}V^u(\mathcal{F},\omega,x,n,\epsilon) \\
    &\leq \liminf_{n\to\infty}-\frac{1}{n}\log{\mu}_{(\omega,x)}^{\xi}((\alpha^{k+n-1}_0)_\omega(x)) \\
    &= \liminf_{n\to\infty}-\frac{1}{n}\log{\mu}_{(\omega,x)}^{\xi}((\alpha^{n-1}_0)_\omega(x)) \\
    &= \liminf_{n\to\infty}\frac{1}{n}I_{{\mu}}(\alpha^{n-1}_0|\xi)(\omega,x).
  \end{align*}
  for ${\mu}$-a.e. $x$.
\end{pf}

The following lemmas are counterparts of those in \cite{HuHuaWu2017}, which are completely parallel to the treatment in \cite{HuHuaWu2017}, so we omit their proofs.
\begin{lemma}\label{lem:SMBlow2}
  Let $\eta\in\mathcal{P}^u(\Omega\times M)$ and $\xi\in\mathcal{Q}^u(\Omega\times M)$. Then for ${\mu}$-a.e. $(\omega,x)$, we have
  \[
    \liminf_{n\to\infty}\frac{1}{n}I_{{\mu}}(\alpha_0^{n-1}|\xi)(\omega,x)=\liminf_{n\to\infty}\frac{1}{n}I_{{\mu}}(\alpha_0^{n-1}|\eta)(\omega,x),
  \]
  \[
    \limsup_{n\to\infty}\frac{1}{n}I_{{\mu}}(\alpha_0^{n-1}|\xi)(\omega,x)=\limsup_{n\to\infty}\frac{1}{n}I_{{\mu}}(\alpha_0^{n-1}|\eta)(\omega,x).
  \]
\end{lemma}
\begin{lemma}\label{lem:SMBup}
  For any $\eta\in\mathcal{P}^u(\Omega\times M)$ and $\xi\in\mathcal{Q}^u(\Omega\times M)$, we have
  \[
    \lim_{n\to\infty}\frac{1}{n}I_{{\mu}}(\Theta^{-n}\xi|\eta)(\omega,x)=\lim_{n\to\infty}\frac{1}{n}I_{{\mu}}(\Theta^{-n}\xi|\xi)(\omega,x)=h_\mu(\mathcal{F},\omega,x,\xi).
  \]
\end{lemma}
\begin{lemma}\label{lem:SMBup2}
  Let $\alpha\in\mathcal{P}(\Omega\times M)$, $\eta\in\mathcal{P}^u(\Omega\times M)$. Then for ${\mu}$-a.e. $(\omega,x)$, we have
  \[
    \lim_{n\to\infty}\frac{1}{n}I_{{\mu}}(\alpha^{n-1}_0|\xi^{n-1}_0\vee\eta)(\omega,x)=0.
  \]
\end{lemma}

Now, we {begin to prove} Theorem \ref{thm:SMB}.

\begin{potB}
  By Lemma \ref{lem:SMBlow} and Lemma \ref{lem:SMBlow2} we can get directly
  \begin{equation}\label{eq:SMBlow}
    h_\mu(\mathcal{F},\alpha|\eta)\leq\liminf_{n\to\infty}\frac{1}{n}I_{{\mu}}(\alpha_0^{n-1}|\eta)(\omega,x).
  \end{equation}
  By Lemma \ref{lem:info}, we have
  \begin{align*}
    I_{{\mu}}(\alpha^{n-1}_0|\eta)(\omega,x) \leq& I_{{\mu}}(\alpha^{n-1}_0\vee\xi^{n-1}_0|\eta)(\omega,x) \\
   =& I_{{\mu}}(\xi^{n-1}_0|\eta)(\omega,x)+I_{{\mu}}(\alpha^{n-1}_0|\xi^{n-1}_0\vee\eta)(\omega,x).
  \end{align*}
  Then by Lemma \ref{lem:SMBup2}, Lemma \ref{lem:SMBup}, and Proposition \ref{prop:localvsfinite2}, we have
  \begin{align}\label{eq:SMBup}
    \limsup_{n\to\infty}\frac{1}{n}I_{{\mu}}(\alpha^{n-1}_0|\eta)(\omega,x) &\leq\limsup_{n\to\infty}\frac{1}{n}I_{{\mu}}(\xi^{n-1}_0|\eta)(\omega,x) \notag \\
    &= h_\mu^u(\mathcal{F})=h_\mu(\mathcal{F},\alpha|\eta).
  \end{align}
  Combining \eqref{eq:SMBlow} and \eqref{eq:SMBup}, we get the result what we need.
\end{potB}

\section{Unstable Pressure}\label{sec:Pressure}

In this section, the definition of unstable pressure for a potential function $\phi\in L^1(\Omega, C(M))$ is given.

Fix $\delta>0$, for $(\omega,x)\in\Omega\times M$, Let $\overline{W^u(\omega,x,\delta)}$ be the $\delta$-neighborhood of $x$ in $W^u(\omega,x)$. A subset $E$ of $\overline{W^u(\omega,x,\delta)}$ is called an $(\omega,n,\epsilon)$ {$W^u$-separated set} if for any $y_1,y_2\in E$, we have $d^u_{\omega,n}(y_1,y_2)>\epsilon$.

It is easy to check that $d^u_{\omega,n}$ is measurable in $\omega$.

Now we can define $P^u(\mathcal{F},\phi,\omega,x,\delta,n,\epsilon)$ as follows,
\begin{align*}
 P^u(\mathcal{F},\phi,\omega,x,\delta,n,\epsilon)=\sup & \Big\{\sum_{y\in E}\exp((S_n\phi)(y)):\\ & \text{ } E\text{ is an }(\omega,n,\epsilon)\ W^u\text{-separated set of }\overline{W^u(\omega,x,\delta)}\Big\},
\end{align*}
{where $(S_n\phi)(y)=\sum_{j=0}^{n-1}\phi(f_\omega^j(y))$.} Set $P^u(\mathcal{F},\phi,\omega,x,\delta,n,\epsilon)=1$, when $(\omega,x)\in\Omega\times M\setminus\Lambda$. Then $P^u(\mathcal{F},\phi,\omega,x,\delta)$ is defined as
\[
  P^u(\mathcal{F},\phi,\omega,x,\delta)=\lim_{\epsilon\to 0}\limsup_{n\to\infty}\frac{1}{n}\log P^u(\mathcal{F},\phi,\omega,x,\delta,n,\epsilon).
\]
Next, we define
\[
  P^u(\mathcal{F},\phi,\omega,\delta)=\sup_{x\in M}P^u(\mathcal{F},\phi,\omega,x,\delta)
\]
and
\[
  P^u(\mathcal{F},\phi,\delta)=\int_{\Omega}P^u(\mathcal{F},\phi,\omega,\delta)\mathrm{d}\mathbf{P}(\omega).
\]
Because $\mathbf{P}$ is ergodic, we have the following lemma.
\begin{lemma}\label{lem:generic}
$P^u(\mathcal{F},\phi,\delta)=P^u(\mathcal{F},\phi,\omega,\delta)$ for $\mathbf{P}$-a.e. $\omega\in\Omega$.
\end{lemma}

Finally, we can give the definition of unstable pressure for $\mathcal{F}$.
\begin{definition}\label{def:upressure}
  The \emph{unstable pressure} for $\mathcal{F}$ is defined as
  \[
    P^u(\mathcal{F},\phi)=\lim_{\delta\to 0}P^u(\mathcal{F},\phi,\delta).
  \]
\end{definition}

\begin{remark}
As what one can do for classical pressure, we can also define the unstable pressure via spanning sets or open covers. We omit the details here.
\end{remark}

\begin{remark}
   {When $\phi\equiv0$}, we call the unstable topological pressure $P^u(\mathcal{F},0)$ {\em the unstable topological entropy} of $\mathcal{F}$, and we denote it by $h^u_{\text{top}}(\mathcal{F})$.
\end{remark}

For the proof of the variational principle, we need the following lemma.
\begin{lemma}\label{lem:pressuremax}
  \[
     {P^u(\mathcal{F},\phi)=P^u(\mathcal{F},\phi,\delta)}\quad\text{for any }\delta>0.
  \]  	
\end{lemma}

\begin{pf}
  It is clear that $P^u(\mathcal{F},\phi)\leq P^u(\mathcal{F},\phi,\delta)$ for any $\delta>0$, since the function $\delta\mapsto P^u(\mathcal{F},\phi,\delta)$ is increasing.

  Now let $\delta>0$ be fixed. For any $\rho>0$ and each $\omega\in\Omega$, there exists $y_\omega$ such that
  \[
    P^u(\mathcal{F},\phi,\omega,\delta)\leq P^u(\mathcal{F},\phi,\omega,y_\omega,\delta)+\frac{\rho}{3}.
  \]
  Take $\epsilon_0>0$ such that
  \[
    P^u(\mathcal{F},\phi,\omega,y_\omega,\delta)\leq \limsup_{n\to\infty}\frac{1}{n}\log P^u(\mathcal{F},\phi,\omega,y_\omega,\delta,n,\epsilon_0)+\frac{\rho}{3}.
  \]
  Then choose $0<\delta_1<\delta$ small enough such that
  \begin{equation}\label{eq:pressuremax}
    P^u(\mathcal{F},\phi)+\frac{\rho}{3}\geq P^u(\mathcal{F},\phi,\delta_1).
  \end{equation}
  There exists a positive number $N=N(\omega)$ which depends on $\delta$, $\delta_1$ and the Riemannian structure on $\overline{W^u(\omega,y_\omega,\delta)}$ such that
  \[
    \overline{W^u(\omega,y_\omega,\delta)}\subset\bigcup_{j=1}^N\overline{W^u(\omega,y_j,\delta_1)}
  \]
  for some $y_j\in\overline{W^u(\omega,y_\omega,\delta)}$, $j=1,2,\cdots,N$. Then we have
  \begin{align*}
    P^u(\mathcal{F},\phi,\omega,\delta)&\leq P^u(\mathcal{F},\phi,\omega,y_\omega,\delta)+\frac{\rho}{3} \\
    &\leq \limsup_{n\to\infty}\frac{1}{n}\log P^u(\mathcal{F},\phi,\omega,y_\omega,\delta,n,\epsilon_0)+\frac{2\rho}{3} \\
    &\leq \limsup_{n\to\infty}\frac{1}{n}\log\Big(\sum_{j=1}^{N}P^u(\mathcal{F},\phi,\omega,y_j,\delta_1,n,\epsilon_0)\Big)+\frac{2\rho}{3} \\
    &\leq \limsup_{n\to\infty}\frac{1}{n}\log N P^u(\mathcal{F},\phi,\omega,y_l,\delta_1,n,\epsilon_0)+\frac{2\rho}{3}\text{ for some }1\leq l\leq N \\
    &= \limsup_{n\to\infty}\frac{1}{n}\log P^u(\mathcal{F},\phi,\omega,y_l,\delta_1,n,\epsilon_0)+\frac{2\rho}{3} \\
    &\leq \lim_{\epsilon\to0}\limsup_{n\to\infty}\frac{1}{n}\log P^u(\mathcal{F},\phi,\omega,y_l,\delta_1,n,\epsilon)+\frac{2\rho}{3} \\
    &\leq P^u(\mathcal{F},\phi,\omega,\delta_1)+\frac{2\rho}{3}.
  \end{align*}
  Integrating both sides of the above inequality, we get
  \[
    P^u(\mathcal{F},\phi,\delta)\leq P^u(\mathcal{F},\phi,\delta_1)+\frac{2\rho}{3}.
  \]
  Thus, by \eqref{eq:pressuremax} we have
  \[
    P^u(\mathcal{F},\phi,\delta)\leq P^u(\mathcal{F},\phi)+\rho.
  \]
  Since $\rho$ is arbitrary, we have
  \[
    P^u(\mathcal{F},\phi,\delta)\leq P^u(\mathcal{F},\phi),
  \]
completing the proof of the lemma.
\end{pf}

The following properties of unstable pressure can be obtained directly from its definition. For $\phi\in L^1(\Omega,C(M))$, we define $\Vert \phi \Vert\colon=\int_{\Omega} \sup_{x\in M}|\phi(\omega,\cdot)| \mathbf{P}(\omega)$.
\begin{proposition}
  Let $\mathcal{F}$ be a $C^2$ {random partially hyperbolic dynamical system}. Then for any $\phi$, $\psi\in L^1(\Omega,C(M))$ and $c\in L^1(\Omega,\mathbf{P})$, the following properties hold.
  \begin{enumerate}[label=(\roman*)]
    \item If $\phi\leq \psi$, then $P^u(\mathcal{F},\phi)\leq P^u(\mathcal{F},\psi)$;
    \item $P^u(\mathcal{F},\phi+c)=P^u(\mathcal{F},\phi)+\int_\Omega c\mathrm{d}\mathbf{P}(\omega)$;
    \item {$h_{\text{top}}^u(\mathcal{F})+\inf \phi \leq P^u(\mathcal{F}, \phi) \leq h_{\text{top}}^u(\mathcal{F})+\sup\phi$;}
    \item {if $P^u(\mathcal{F},\cdot)<\infty$, $|P^u(\mathcal{F}, \phi)-P^u(\mathcal{F}, \psi)|\leq \|\phi-\psi\|$;}
    \item {if $P^u(\mathcal{F},\cdot)<\infty$}, then the map $P^u(\mathcal{F},\cdot)\colon L^1(\Omega,C(M))\to\mathbb{R}\cup\{\infty\}$ is convex;
    \item $P^u(\mathcal{F},\phi+\psi\circ\Theta-\psi)=P^u(\mathcal{F},\phi)$;
    \item $P^u(\mathcal{F},\phi+\psi)\leq P^u(\mathcal{F},\phi)+P^u(\mathcal{F},\psi)$.
  \end{enumerate}
\end{proposition}
\section{A Variational Principle}\label{sec:VP}

In this section, we give the proof of Theorem \ref{thm:vp}. Firstly, we give the following well-known lemma,  {which is almost identical to Lemma 1.24 in \cite{Bowen1975} except that we have removed the condition $s \leq 1$.}

\begin{lemma}\label{lem:wellknown}
  Suppose $0\leq p_u$, $\cdots$, $p_m\leq1$, {$s=p_1+\cdots+p_m$} and $a_1$, $\cdots$, $a_m\in\mathbb{R}$. Then
  \[
    \sum_{u=1}^{m}p_u(a_u-\log p_u)\leq s\left(\log\sum_{u=1}^{m}e^{a_u}-\log s\right).
  \]
\end{lemma}

\begin{proposition}\label{prop:vpneq}
  For $\mu\in\mathcal{M}_{\mathbf{P}}(\mathcal{F})$,
  \[
    h^u_\mu(\mathcal{F})+\int_{\Omega\times M}\phi \mathrm{d}{\mu}\leq P^u(\mathcal{F},\phi).
  \]
\end{proposition}

\begin{pf}
 Let $\mu=\int_{\mathcal{M}^e_{\mathbf{P}}(\mathcal{F})}\nu d\tau(\nu)$ be the unique ergodic decomposition
where $\mathcal{M}^e_{\mathbf{P}}(\mathcal{F})$ is the set of ergodic measures in $\mathcal{M}_{\mathbf{P}}(\mathcal{F})$ and $\tau$ is a Borel probability measure such that $\tau(\mathcal{M}^e_{\mathbf{P}}(\mathcal{F}))=1$.
Since $\mu \mapsto h_\mu^u(\mathcal{F})$ is affine and upper semi-continuous by Lemma \ref{lem:affine} and \ref{lem:semiconti}, then so is $\mu \mapsto h_\mu^u(\mathcal{F})+\int_{\Omega\times M}\phi \mathrm{d}{\mu}$ and hence
\begin{equation*}\label{e:ergodicdecom}
h_\mu^u(\mathcal{F})+\int_{\Omega\times M}\phi \mathrm{d}{\mu}=\int_{\mathcal{M}^e_{\mathbf{P}}(\mathcal{F})}\Big(h_\nu^u(\mathcal{F})+\int_{\Omega\times M}\phi \mathrm{d}{\nu}\Big) d\tau(\nu)
\end{equation*}
by a classical result in convex analysis (cf. Fact A.2.10 on p. 356 in \cite{Downarowicz2011}). So we only need to prove the proposition for
ergodic measures.

 We assume $\mu$ is ergodic. Let $\xi\in \mathcal{Q}^u(\Omega\times M)$, that is, a measurable partition of $\Omega\times M$ subordinate to $W^u$-foliation as in Proposition \ref{prop:specialpartition}. Then we can pick $(\omega,x)\in\Omega\times M$ satisfying
  \begin{enumerate}[label=(\roman*)]
    \item ${\mu}_{(\omega,x)}^\xi(\xi_\omega(x))=1$;
    \item\label{prp:item2} there exists $B\subset\xi_\omega(x)$ such that
    \begin{enumerate}
      \item ${\mu}_{(\omega,x)}^\xi(B)=1$,
      \item $h_\mu(\mathcal{F},\omega,\xi)=h_\mu(\mathcal{F},\omega,y,\xi)= \lim_{n\to\infty}-\frac{1}{n}\log{\mu}^\xi_{(\omega,y)}(V^u(\mathcal{F},\omega,y,n,\epsilon))$ for any $y\in B$ and $\epsilon>0$, according to Lemma \ref{lem:low=up},
      \item $\lim_{n\to\infty}\frac{1}{n}(S_n\phi)(\omega,y)=\int_{\Omega\times M}\phi \mathrm{d}{\mu}$ for any $y\in B$, which can be obtained by using the Birkhoff ergodic theorem on $(\Omega\times M,\Theta)$.
    \end{enumerate}
       \end{enumerate}

  Fix $\rho>0$. By property \ref{prp:item2} we know that for any $y\in B$, there exists $N(y)=N(y,\epsilon)>0$ such that if $n\geq N(y)$ then we have
  \[
    {\mu}^\xi_{(\omega,y)}(V^u(\mathcal{F},\omega,y,n,\epsilon))\leq e^{-n(h_\mu(\mathcal{F},\omega,\xi)-\rho)}
  \]
  and
  \begin{equation}\label{eq:est2}
    \frac{1}{n}(S_n\phi)(\omega,y)\geq\int_{\Omega\times M}\phi \mathrm{d}{\mu}-\rho.
  \end{equation}
  Denote $B_n=\{y\in B\colon N(y)\leq n\}$. Then $B=\bigcup_{n=1}^\infty B_n$. So we can choose $n>0$ such that ${\mu}^\xi_{(\omega,x)}(B_n)>{\mu}^\xi_{(\omega,x)}(B)-\rho=1-\rho$. If $y\in B_n\subset\xi_\omega(x)$, then ${\mu}^\xi_{(\omega,y)}={\mu}^\xi_{(\omega,x)}$. So for any $y\in B_n$ we have
  \begin{equation}\label{eq:est3}
    {\mu}^\xi_{(\omega,x)}(V^u(\mathcal{F},\omega,y,n,\epsilon))\leq e^{-n(h_\mu(\mathcal{F},\omega,\xi)-\rho)}.
  \end{equation}
  Now we can choose $\delta>0$ such that $W^u(\omega,x,\delta)\supset\xi_\omega(x)$. Let $F$ be an {$(\omega, n,\epsilon/2)$ $W^u$-spanning set} of $\overline{W^u(\omega,x,\delta)}\cap B_n$ satisfying
  \[
    \overline{W^u(\omega,x,\delta)}\cap B_n\subset\bigcup_{z\in F}V^u(\mathcal{F},\omega,z,n,\epsilon/2),
  \]
  and $V^u(\mathcal{F},\omega,z,n,\epsilon/2)\cap B_n\neq\emptyset$ for any $z\in F$. Then choose an arbitrary point in $V^u(\mathcal{F},\omega,z,n,\epsilon/2)\cap B_n$, which is denoted by $y(z)$. Then we have
  \begin{align}\label{eq:est3.5}
    1-\rho &< {\mu}^\xi_{(\omega,x)}(\overline{W^u(\omega,x,\delta)}\cap B_n)\notag\\
    &\leq {\mu}^\xi_{(\omega,x)}(\bigcup_{z\in F}V^u(\mathcal{F},\omega,z,n,\epsilon/2)) \notag\\
    &\leq \sum_{z\in F}{\mu}^\xi_{(\omega,x)}(V^u(\mathcal{F},\omega,z,n,\epsilon/2)) \notag\\
    & \leq\sum_{z\in F}{\mu}^\xi_{(\omega,x)}(V^u(\mathcal{F},\omega,y(z),n,\epsilon)).
  \end{align}
  Using \eqref{eq:est2}, \eqref{eq:est3} and Lemma \ref{lem:wellknown} with
  \[
    p_i={\mu}^\xi_{(\omega,x)}(V^u(\mathcal{F},\omega,y(z),n,\epsilon))\text{ and }a_i=(S_n\phi)(\omega,y(z)),
  \]
  we have
  \begin{align*}
    &\sum_{z\in F}{\mu}^\xi_{(\omega,x)}(V^u(\mathcal{F},\omega,y(z),n,\epsilon))\left(n\left(\int_{\Omega\times M}\phi \mathrm{d}{\mu}-\rho\right)+n(h_\mu(\mathcal{F},\omega,\xi)-\rho)\right)\notag\\
\leq &\sum_{z\in F}{\mu}^\xi_{(\omega,x)}(V^u(\mathcal{F},\omega,y(z),n,\epsilon))\Big((S_n\phi)(y(z))-\log {\mu}^\xi_{(\omega,x)}(V^u(\mathcal{F},\omega,y(z),n,\epsilon))\Big)\notag\\
\leq &\left(\sum_{z\in F}{\mu}^\xi_{(\omega,x)}(V^u(\mathcal{F},\omega,y(z),n,\epsilon))\right)\left(\log \sum_{z\in F}\exp((S_n\phi)(y(z)))-\right. \\
\qquad& \log \left.\sum_{z\in F}{\mu}^\xi_{(\omega,x)}(V^u(\mathcal{F},\omega,y(z),n,\epsilon))\right).
  \end{align*}
  Combining  \eqref{eq:est3.5},
  \begin{align}\label{eq:est4}
    & n\Big(\int_{\Omega\times M}\phi \mathrm{d}{\mu}-\rho\Big)+n(h_\mu(\mathcal{F},\omega,\xi)-\rho) \notag\\
    \leq &\log \sum_{z\in F}\exp((S_n\varphi)(y(z)))-\log \sum_{z\in F}{\mu}^\xi_{(\omega,x)}(V^u(\mathcal{F},\omega,y(z),n,\epsilon))\notag\\
    \leq & \log\sum_{z\in F}\exp((S_n\phi)(\omega,y(z)))-\log(1-\rho).
  \end{align}
  Let ${\Delta_{\omega,\epsilon}:=\sup\{|\phi(\omega,x)-\phi(\omega,y)|\colon d(x,y)\leq\epsilon\}}$. For any $z\in F$, we have
  \[
    \exp((S_n\phi)(\omega,y(z)))\leq\exp((S_n\phi)(\omega,z)+n\Delta_{\omega,\epsilon}).
  \]
  Dividing by $n$ and taking the $\limsup$ on both side of \eqref{eq:est4}, we have
  \[
    \int_{\Omega\times M}\phi \mathrm{d}{\mu}+h_\mu(\mathcal{F},\omega,\xi)-2\rho\leq\limsup_{n\to\infty}\frac{1}{n}\log\sum_{z\in F}\exp((S_n\phi)(\omega,z))+\Delta_{\omega,\epsilon}.
  \]
  We can choose a sequence $\{F_n\}$ of such $F$ such that
  \[
    \limsup_{n\to\infty}\frac{1}{n}\log\sum_{z\in F_n}\exp((S_n\phi)(\omega,z))\leq {P^u(\mathcal{F},\phi,\omega,\delta)}.
  \]
    Since $\rho$ is arbitrary, and $\Delta_{\omega,\epsilon}\to 0$ as $\epsilon\to0$, we have
  \[
     \int_{\Omega\times M}\phi \mathrm{d}{\mu}+h_\mu(\mathcal{F},\omega,\xi)\leq P^u(\mathcal{F},\phi,\omega,\delta).
  \]
  Integrating with respect to $\omega$ gives what we need.
\end{pf}

\begin{potC}
  By Proposition \ref{prop:vpneq}, we only need to prove that for any $\rho>0$, there exists $\mu\in\mathcal{M}_{\mathbf{P}}(\mathcal{F})$ such that $h^u_\mu(\mathcal{F})+\int_{\Omega\times M}\phi \mathrm{d}{\mu}\geq P^u(\mathcal{F},\phi)-\rho$.

  Given $\delta>0$, by Lemma \ref{lem:generic} and Birkhoff Ergodic Theorem, we can choose $\omega_0\in\Omega$ such that
  \[
    P^u(\mathcal{F},\phi,\delta)=P^u(\mathcal{F},\phi,\omega_0,\delta),
  \]
  and
  \[
    \lim_{n\to \infty}\frac{1}{n}\sum_{u=0}^{n-1}K(\theta^n\omega_0)=\int_{\Omega}K(\omega)d\mathbf{P}(\omega).
  \]
  Then we can choose $x_0\in M$ such that
  \[
    P^u(\mathcal{F},\phi,\omega_0,x_0,\delta)\geq P^u(\mathcal{F},\phi,\omega_0,\delta)-\rho.
  \]
Take $\epsilon>0$ small enough. Then let $E_{\omega_0,n}$ be an $(\omega_0,n,\epsilon)$ $W^u$-separated set of $\overline{W^u(\omega_0,x_0,\delta)}$ such that
  \[
    \log\sum_{y\in E_{\omega_0,n}}\exp((S_n\phi)(\omega_0,y))\geq\log  P^u(\mathcal{F},\phi,\omega_0,x_0,\delta,n,\epsilon)-1.
  \]
  Then we construct measures ${\nu}_{n}$ with support $\{\omega_0\}\times M$ such that
  \[
    \mathrm{d}{\nu}_{n}(\omega_0,x)=\mathrm{d}\nu^{(n)}_{\omega_0}(x)\mathrm{d}\delta_{\omega_0}(\omega),
  \]
  where
  \[
    \nu^{(n)}_{\omega_0}:=\frac{\sum_{y\in E_{\omega_0,n}}\exp((S_n\phi)(\omega_0,y))\delta_y}{\sum_{z\in E_{\omega_0,n}}\exp((S_n\phi)(\omega_0,z))}
  \]
  and $\delta_\cdot$ denotes a Dirac measure.
  Let
  \[
    {\mu}_{n}=\frac{1}{n}\sum_{i=0}^{n-1}\Theta^i{\nu}_{n}.
  \]
  Then by Lemma 2.1 in \cite{Kifer2001}, there exists a subsequence $\{n_i\}$ such that
  \[
    \lim_{i\to\infty}{\mu}_{n_i}={\mu}.
  \]
  It is easy to check that $\mu\in\mathcal{M}_{\mathbf{P}}(\mathcal{F})$.

  We can choose a partition $\eta\in\mathcal{P}^u(\Omega\times M)$ such that $\overline{W^u(\omega_0,x_0,\delta)}\subset \eta_{\omega_0}(x_0)$ (by shrinking $\delta$ if necessary). That is, $\overline{W^u(\omega_0,x_0,\delta)}$ is contained in a single element of $\eta_{\omega_0}$. Then choose a fiberwise finite partition $\alpha$ of $\Omega\times M$ with sufficiently small diameter such that ${\mu}_\omega(\partial \alpha_\omega)=0$ for $\mathbf{P}$-a.e. $\omega$. {Let $\alpha^u$ denote the corresponding measurable partition in $\mathcal{P}^u(\Omega\times M)$ constructed via $\alpha$.}

  Fix $q$, $n\in\mathbb{N}$ with $1<q\leq n-1$. Put $a(j)=\left[\frac{n-j}{q}\right]$, $j=0,1,\cdots,q-1$, where we denote by $[a]$ the integer part of $a$. Then
    \[\bigvee_{u=0}^{n-1}\Theta^{-i}\alpha=\bigvee_{r=0}^{a(j)-1}\Theta^{-(rq+j)}\alpha_0^{q-1}\vee\bigvee_{t\in T_j}\Theta^{-t}\alpha,
  \]
  where $T_j=\{0,1,\cdots,j-1\}\cup\{j+aq(j),\cdots,n-1\}$. Note that $\mathrm{Card\ } T_j\leq 2q$. {For $\mathbf{P}$-a.e. $\omega\in\Omega$, suppose that $\alpha_\omega$ contains $K(\omega)$ elements, moreover, we require that $\mathrm{diam}(\alpha_\omega)\ll\epsilon$.} Then
  \begin{align*}
      & \log\sum_{y\in E_{\omega_0,n}}\exp((S_n\phi)(\omega_0,y)) \\
     =& \sum_{y\in E_{\omega_0,n}}\nu^{(n)}_{\omega_0}(\{y\})\Big(-\log\nu^{(n)}_{\omega_0}(\{y\})+(S_n\phi)(\omega_0,y)\Big) \\
     =& H_{\nu_n}(\alpha^{n-1}_0|\eta)+\int_{\Omega\times M}(S_n\phi)\mathrm{d}{\nu}_n.
  \end{align*}
 Then following the same calculation in \cite{HuWuZhu2017}, we have that
  \begin{align*}
         & \log\sum_{y\in E_{\omega_0,n}}\exp((S_n\phi)(\omega_0,y)) \\
    \leq & \sum_{t\in T_j}H_{{\nu}_n}(\Theta^{-t}\alpha|\eta)+H_{\Theta^j{\nu}_n}(\alpha^{q-1}_0|\Theta^j\eta) \\
    +    &\sum_{r=1}^{a(j)-1}H_{\Theta^{rq+j}{\nu}_n}(\alpha^{q-1}_0|\Theta\alpha^u)+\int_{\Omega\times M}(S_n\phi)\mathrm{d}{\nu}_n\\
    \leq & 2q\log K_n(\omega_0)+H_{\Theta^j{\nu}_n}(\alpha^{q-1}_0|\Theta^j\eta) \\
    +    &\sum_{r=1}^{a(j)-1}H_{\Theta^{rq+j}{\nu}_n}(\alpha^{q-1}_0|\Theta\alpha^u)+\int_{\Omega\times M}(S_n\phi)\mathrm{d}{\nu}_n
  \end{align*}
where $K_n(\omega_0):=\max_{t\in T_j}K(\theta^t\omega)$. We claim that $\lim_{n\to \infty}\frac{1}{n}\log K_n(\omega_0)=0$. Indeed, by the choice of $\omega_0$, we know that
  $$\lim_{n\to \infty}\frac{1}{n}\sum_{u=0}^{n-1}K(\theta^n\omega_0)=\int_{\Omega}K(\omega)d\mathbf{P}(\omega)<\infty$$
  as $\alpha$ is fiberwise finite. So $\lim_{n\to \infty}\frac{1}{n}K(\theta^n\omega_0)=0$, from which the claim follows easily.

 Summing the inequality above over $j$ from $0$ to $q-1$ and dividing by $n$, by Lemma \ref{lem:affine} we have
  \begin{align}\label{eq:keyest}
         & \frac{q}{n}\log\sum_{y\in E_{\omega_0,n}}\exp((S_n\phi)(\omega_0,y)) \notag \\
    \leq & \frac{2q^2}{n}\log K_n(\omega_0)+\frac{1}{n}\sum_{j=0}^{q-1}H_{\Theta^j{\nu}_n}(\alpha^{q-1}_0|\Theta^j\eta) \notag  \\
    +    & \frac{1}{n}\sum_{k=0}^{n-1}H_{\Theta^k{\nu}_n}(\alpha^{q-1}_0|\Theta\alpha^u)+\frac{q}{n}\int_{\Omega\times M}(S_n\phi)\mathrm{d}{\nu}_n \notag  \\
    \leq & \frac{2q^2}{n}\log K_n(\omega_0)+\frac{1}{n}\sum_{j=0}^{q-1}H_{\Theta^j{\nu}_n}(\alpha^{q-1}_0|\Theta^j\eta) \notag  \\
    +    & H_{{\mu}_n}(\alpha^{q-1}_0|\Theta\alpha^u)+q\int_{\Omega\times M}\phi \mathrm{d}{\mu}_n.
  \end{align}
  Then we can choose a sequence $\{n_k\}$ such that
  \begin{enumerate}[label=(\roman*)]
    \item ${\mu}_{n_k}\to{\mu}$ as $k\to\infty$;
    \item the following equality holds
    \begin{align*}
     & \lim_{k\to\infty}\frac{1}{n_k}\log P^u(\mathcal{F},\phi,\omega_0,x_0,\delta,n_k,\epsilon) \\
    =& \limsup_{n\to\infty}\frac{1}{n}\log P^u(\mathcal{F},\phi,\omega_0,x_0,\delta,n,\epsilon);
    \end{align*}
    \item ${\nu}_{n_k}\to{\nu}$ as $k\to\infty$ for some measure on $\Omega\times M$.
  \end{enumerate}
  Since ${\mu}_\omega(\partial \alpha_\omega)=0$ for $\mathbf{P}$-a.e. $\omega$, by Lemma \ref{lem:semiconti},
  \[
    \limsup_{k\to\infty}H_{{\mu}_{n_k}}(\alpha_0^{q-1}|\Theta\alpha^u)\leq H_{{\mu}}(\alpha_0^{q-1}|\Theta\alpha^u).
  \]
  As $\tilde\nu_{n}$ is supported on $\{\omega_0\}\times\overline{W^u(\omega_0,x_0,\delta)}$, for each $j=0,\cdots,q-1$, we can choose $\alpha, \beta_n \in \mathcal{P}(\Omega \times M)$ such that $\beta_1<\beta_2<\cdots<\beta_n<\cdots$ and {$\mathscr{B}(\beta_n)\nearrow\mathscr{B}(\Theta^j\eta)$}, and moreover, $(\Theta^j{\nu})_\omega(\partial(\alpha_0^{q-1})_\omega)=0$, $(\Theta^j{\nu})_\omega((\partial(\beta_n)_0^{q-1})_\omega)=0$ for $\mathbf{P}$-a.e. $\omega\in\Omega$. Then applying Lemma \ref{lem:semiconti} we have
  \[
    \limsup_{k\to\infty}\frac{1}{n_k}\sum_{j=0}^{q-1}H_{\Theta^j{\nu}_{n_k}}(\alpha_0^{q-1}|\Theta^j\eta)\leq \limsup_{k\to\infty}\frac{1}{n_k}\sum_{j=0}^{q-1}H_{\Theta^j{\nu}}(\alpha_0^{q-1}|\Theta^j\eta)=0.
  \]
  Thus replacing $n$ by $n_k$ in \eqref{eq:keyest} and letting $k\to\infty$, by the above claim and discussions, we get
  \begin{align*}
        & q\limsup_{n\to\infty}\frac{1}{n}\log P^u(\mathcal{F},\phi,\omega_0,x_0,\delta,n,\epsilon) \\
    \leq& H_{{\mu}}(\alpha^{q-1}_0|\Theta\alpha^u)+q\int_{\Omega\times M}(S_n\phi)\mathrm{d}{\mu}.
  \end{align*}
  By Theorem \ref{thm:localvsfinite},
  \begin{align*}
        & \limsup_{n\to\infty}\frac{1}{n}\log P^u(\mathcal{F},\phi,\omega_0,x_0,\delta,n,\epsilon) \\
    \leq& \lim_{q\to\infty}\frac{1}{q}H_{{\mu}}(\alpha^{q-1}_0|\Theta\alpha^u)+\int_{\Omega\times M}(S_n\phi)\mathrm{d}{\mu} \\
       =   & h^u_{\mu}(\mathcal{F})+\int_{\Omega\times M}(S_n\phi)\mathrm{d}{\mu}.
  \end{align*}
  {Letting} $\epsilon\to0$, we have $P^u(\mathcal{F},\phi,\omega_0,x_0,\delta)\leq  h^u_{\mu}(\mathcal{F})+\int_{\Omega\times M}(S_n\phi)\mathrm{d}{\mu}$. Recall that $P^u(\mathcal{F},\phi)=P^u(\mathcal{F},\phi,\delta)=P^u(\mathcal{F},\phi,\omega_0, \delta)\leq P^u(\mathcal{F},\phi,\omega_0,x_0,\delta)+\rho$. The proof of Theorem \ref{thm:vp} is complete.
\end{potC}

\section{\texorpdfstring{$u$}{u}-Equilibrium States}\label{sec:equilibrium}
In this section, as an application of Theorem \ref{thm:vp}, we consider the $u$-equilibrium states, which is given in Definition \ref{def:uequi}.

By Lemma \ref{lem:semiconti}, we know that $\mu\mapsto h_{\mu}^u(\mathcal{F})$ is upper semi-continuous at $\mu_0\in\mathcal{M}_{\mathbf{P}}(\mathcal{F})$, then by the similar technique to the proof of Theorem 3.1.6 in \cite{Bogenschuetz1993}, we can prove a dual proposition of Theorem \ref{thm:vp}.
\begin{proposition}\label{prop:dualvp}
  If $h^u_{\text{top}}(\mathcal{F})<\infty$ and $\mu_0\in\mathcal{M}_{\mathbf{P}}(\mathcal{F})$, then
  \[
    h_{\mu_0}^u(\mathcal{F})=\inf_{\phi\in L^1(\Omega, C(M))}\left\{P^u(\mathcal{\mathcal{F},\phi})-\int_{\Omega\times M}\phi\mathrm{d}\tilde{\mu_0}\right\}.
  \]
\end{proposition}

Consider $\phi\in L^1(\Omega, C(M))$.
\begin{definition}\label{def:uequi}
  $\mu\in\mathcal{M}_{\mathbf{P}}(\mathcal{F})$ is said to be a $u$-equilibrium state for $\phi$, if it satisfies
  \[
  h^u_\mu(\mathcal{F})+\int_{\Omega\times M}\phi\mathrm{d}{\mu}=P^u(\mathcal{F},\phi).
  \]
\end{definition}
We denote by $\mathcal{M}^u(\mathcal{F},\phi)$ the set of all $u$-equilibrium states for $\phi$. By the completely parallel treatment in \cite{ZhuLiLi2009}, we have the following proposition.
\begin{proposition}\label{prop:equilibrium}
  Let $\phi\in L^1(\Omega, C(M))$, then we have the following properties on $u$-equilibrium states.
  \begin{enumerate}[label=(\roman*)]
    \item\label{equi1} $\mathcal{M}^u(\mathcal{F},\phi)$ is non-empty, and it is convex;
    \item the extreme points of $\mathcal{M}^u(\mathcal{F},\phi)$ are precisely ergodic members of $\mathcal{M}^u(\mathcal{F},\phi)$;
    \item $\mathcal{M}^u(\mathcal{F},\phi)$ is compact and has an {ergodic $u$-equilibrium state};
    \item assume $\phi$, $\psi\in L^1(\Omega,C(M))$ are cohomologous, i.e. $\phi=\psi+\sigma-\sigma\circ\Theta-c$ for some $c\in L^1(\Omega,\mathbf{P})$ and $\sigma\in L^1(\Omega, C(M))$. Then $\phi$ and $\psi$ have the same $u$-equilibrium states, and
        \[
          P^u(\mathcal{F},\phi)=P^u(\mathcal{F},\psi)-\int_{\Omega} c\mathrm{d}\mathbf{P}(\omega).
        \]
  \end{enumerate}
\end{proposition}

\begin{remark}
  By \ref{equi1} in Proposition \ref{prop:equilibrium}, we know that for $\phi\equiv0$, $\mathcal{M}^u(\mathcal{F},0)$ is not empty, which means that the maximal unstable metric entropy always exists.
\end{remark}

There are some important potential functions interest people, among which, $\phi^u(\omega,x)=-\log\Vert\mathrm{det}Df_\omega|E^u(\omega,x)\Vert$ is crucial. Indeed, $\phi^u$ is closely related to the so-called Gibbs $u$-states as follows.
\begin{definition}\label{def:Gibbs}
  $\mu\in\mathcal{M}_{\mathbf{P}}(\mathcal{F})$ is called a \textit{Gibbs $u$-state} if $\mu^\xi_{(\omega,x)}\ll\lambda_{(\omega,x)}$, for $\mu$-a.e. $(\omega,x)\in\Omega\times M$, where $\xi$ is an increasing partition subordinate to $W^u$-foliation, $\{\mu^\xi_{(\omega,x)}\}$ is the corresponding canonical system of conditional measures and $\lambda_{(\omega,x)}$ is the Lebesgue measure on $W^u(\omega,x)$ induced by its inherited Riemannian structure as a submanifold.
\end{definition}

By results in \cite{BahnmullerLiu1998}, \cite{QianQianXie2003}, \cite{LedrappierYoung1988}, \cite{LedrappierYoung1988a}, for $\mu\in\mathcal{M}_{\mathbf{P}}(\mathcal{F})$, we have
\begin{equation}\label{eq:Gibbs}
  h^u_\mu(\mathcal{F})\leq\int_{\Omega\times M}-\phi^u\mathrm{d}\mu,
\end{equation}
and the equality holds if and only if $\mu$ is a Gibbs $u$-state.

As an application of Theorem \ref{thm:vp}, we give a characterization of Gibbs $u$-states as follows.
\begin{proposition}
  $\mu\in\mathcal{M}_{\mathbf{P}}(\mathcal{F})$ is a Gibbs $u$-state if and only if $\mu$ is a $u$-equilibrium state of $\phi^u$.
\end{proposition}
\begin{proof}
  By Theorem \ref{thm:vp} and Definition \ref{def:Gibbs}, we have
  \[
    P^u(\mathcal{F},\phi^u)=\sup_{\mu\in\mathcal{M}_{\mathbf{P}}(\mathcal{F})}\left\{h^u_\mu(\mathcal{F})+\int_{\Omega\times M}\phi\mathrm{d}\mu\right\}=0,
  \]
  and $\mu$ is a Gibbs $u$-state if and only if $\mu$ is a $u$-equilibrium state of $\phi^u$.
\end{proof}

\begin{remark}
  \begin{enumerate}[fullwidth,label=(\roman*)]
    \item From the proof of above proposition, we obtain that $P^u(\mathcal{F},\phi^u)=0$.
    \item By Proposition \ref{prop:equilibrium} (i), we know that the Gibbs $u$-state always exists, which is a result of independent interest.
  \end{enumerate}
\end{remark}

\end{document}